\documentclass[12pt]{amsart}
\usepackage{amsmath,amscd, nicefrac, xcolor, setspace, amsfonts, amssymb, bbm, amsthm, tikz, tikz-cd, enumerate, graphicx,verbatim,todonotes,mathtools,url}
\usetikzlibrary{calc,decorations.markings}
\usepackage[margin=1in]{geometry}
\usepackage[shortalphabetic]{amsrefs}
\graphicspath{ {/images/} }
\usepackage{eqnarray,amsmath}
\usepackage{cite}
\usepackage{amsmath,amssymb,amsfonts}
\usepackage{esint}
\usepackage{graphicx}
\usepackage{algorithm,algpseudocode}
\usepackage{textcomp}
\usepackage{bbm}
\usepackage{comment}
\usepackage{mathtools}

\newtheorem{theorem}{Theorem}[section]

\newtheorem{proposition}{Proposition}[section]

\newtheorem{example}{Example}[section]
\newtheorem{remark}{Remark}[section]

\newcommand{\R}{\mathbb{R}}

\newcommand{\map}[3]{#1:#2 \rightarrow #3}

\DeclarePairedDelimiterX{\infdivx}[2]{(}{)}{%
  #1\;\delimsize\|\;#2%
}
\newcommand{\infdiv}{\mathbb{D}\infdivx}

\numberwithin{equation}{section}

\title{
Schrodinger Bridge over Averaged Systems
}

\author{Daniel Owusu Adu and Yongxin Chen
}

\begin{document}

\maketitle
\thispagestyle{empty}
\pagestyle{empty}

\begin{abstract}
We consider a Schrödinger bridge problem where the Markov process is subject to parameter perturbations, forming an ensemble of systems. Our objective is to steer this ensemble from the initial distribution to the final distribution using controls robust to the parameter perturbations. Utilizing the path integral formalism, we demonstrate that the optimal control is a non-Markovian strategy, specifically a stochastic feedforward control, which depends on past and present noise. This unexpected deviation from established strategies for Schrödinger bridge problems highlights the intricate interrelationships present in the system's dynamics. From the perspective of optimal transport, a significant by-product of our work is the demonstration that, when the evolution of a distribution is subject to parameter perturbations, it is possible to robustly deform the distribution to a desired final state using stochastic feedforward controls.
\end{abstract}


\section{Introduction} 
\label{sec:introduction}
This paper concerns the problem of conditioning a process at two endpoints. This problem was first studied by Schr\"{o}dinger in~\cite{SE:31}. For precision, we assume that all random variables are defined over an underlying probability space with probability $\mathbb{P}$. We postulate the Schr\"{o}dinger bridge problem in its generality as follows; assume some \emph{fully observed} particles distributed according to an initial probability measure $\mu_0\in\mathcal{P}(\R^d)$ with density $\rho_0$ at time $t=0$, evolve according to a process $\{y(t)\}_{0\leq t\leq t_f}$ in $\R^d$ with transition density $q$.
 Then, according to the law of large numbers, the final density $P(\cdot,t_f)$ is determined by 
 \[
P(x_f,t_f)=\int_{\R^d}q(0,x_0,t_f,x_f)\rho_0(x_0)dx_0.
 \]
 Suppose that at time $t=t_f$ the observed particles are distributed according to $\mu_f\in\mathcal{P}(\R^d)$ with density $\rho_f$, where $$\rho_f(x_f)\neq P(x_f,t_f).$$ Then, $\mu_f$ deviates from the law of large numbers. This means that our assumption of the process $\{y(t)\}_{0\leq t\leq t_f}$ is inaccurate. The following question arises:
\begin{enumerate}
    \item What density $\hat{q}$ satisfies 
\[
\rho_f(x_f)=\int_{\R^d}\hat{q}(0,x_0,t_f,x_f)\rho_0(x_0)dx_0.
\]
\item Among such densities $\hat{q}$, which one is closest to $q$ in some suitable sense.
\end{enumerate}

Statement~1 and~2 constitute the Schr\"{o}dinger bridge problem and the most likely stochastic process denoted as $\{x(t)\}_{0\leq t\leq t_f}$ such that the distributions of $x(0)$ and $x(t_f)$ coincide with $\mu_0$ and $\mu_f$, respectively, is called the Schr\"{o}dinger bridge.  The Schr\"{o}dinger bridge problem reduces to finding the most likely probability law $\mathbf{P}$ of $x(\cdot)$ admitting the end-time marginals $\mu_0$ and $\mu_f$ with densities  $\rho_0$ and $\rho_f$, respectively~\cite{HF:88,YC-TTG-MP:21,CL:13}. More precisely the problem
\begin{equation}\label{eq:original dynamic Schrodinger bridge problem}
\min_{\mathbf{P}}\infdiv{ \mathbf{P}}{\mathbf{R}}:=\int_{\mathcal{X}}\mathrm{log}\left(\frac{d \mathbf{P}}{d \mathbf{R}}\right)d \mathbf{P},
\end{equation}
where $\mathcal{X}:=C([0,t_f];\R^d)$ is the space of all continuous $\R^d$-valued paths on $[0,t_f]$, $\mathbf{R}(\cdot)=\mathbb{P}(y\in(\cdot))$ is the probability law of the process $\{y(t)\}_{0\leq t\leq t_f}$ and the minimization is taken over the space 
of  probability law $\mathbf{P}$ of the process $\{x(t)\}_{0\leq t\leq t_f}$  such that $\mathbf{P}_{0}(\cdot):=\mathbb{P}(x(0)\in(\cdot))=\mu_0(\cdot)$ and $\mathbf{P}_{f}(\cdot):=\mathbb{P}(x(t_f)\in(\cdot))=\mu_f(\cdot)$, where $\mu_0,\mu_f\in\mathcal{P}(\R^d)$ have densities  $\rho_0$ and $\rho_f$, respectively.
Note that~\eqref{eq:original dynamic Schrodinger bridge problem} well-defined if $\frac{d \mathbf{P}}{d \mathbf{R}}$ is the Radon-Nikodym derivative. By disintegrating the probability law  $\mathbf{P}$, the Schr\"{o}dinger problem reduces to the static problem: 
\begin{equation}\label{eq:original Schrodinger bridge problem}
\min_{\mathbf{P}_{0,f}\in\Pi(\mu_0,\mu_f)}\int_{\R^d\times\R^d}\mathrm{log}\left(\frac{d \mathbf{P}_{0,f}}{d q(0,\cdot,t_f,\cdot)}\right)d \mathbf{P}_{0,f},
\end{equation}
where
 \begin{equation*}
 \Pi(\mu_0,\mu_f):=\{\mathbf{P}_{0,f}\in\mathcal{P}(\R^d\times \R^d) : \mathbf{P}_{0,f}\circ\pi^{-1}_0=\mu_0\quad\text{and}\quad \mathbf{P}_{0,f}\circ\pi^{-1}_f=\mu_f\},     
 \end{equation*}
 where $\pi_0(x_0,x_f)=x_0$ and $\pi_f(x_0,x_f)=x_f$, for all $(x_0,x_f)\in\R^d\times\R^d$. Here, we have identified the densities and the measures are the same. Since~\eqref{eq:original Schrodinger bridge problem} only makes sense whenever $\mathbf{P}_{0,f}$ is absolutely continuous with respect to $q$,  for later use, the unique optimizer of~\eqref{eq:original Schrodinger bridge problem} is characterized by 
 \begin{equation}\label{eq: opt_joint_dis}
 \mathbf{P}_{0,f}(A)=\int_A \phi_0(x_0)q(0,x_0,t_f,x_f)\phi_f(x_f)d x_0d x_f,    
 \end{equation}
for all measurable subset $A\subset\R^d\times\R^d$,  where the pair $(\phi_0,\phi_f)$ satisfies
 \begin{align}\label{eq: Schrod_system_cl_c}
\rho_0(x_0)=&\phi_0(x_0)\int_{\R^d}q(0,x_0,t_f,x_f)\phi_f(x_f)d x_f,\cr
\rho_f(x_f)=&\phi_f(x_f)\int_{\R^d}q(0,x_0,t_f,x_f)\phi_0(x_0)d x_0.
\end{align}
 For more details on the Schr\"{o}dinger bridge problem and its equivalency with the optimal transport problem, we refer the reader to~\cite{YC-TTG-MP:21,CL:13} for instance.

In this paper, we study the following stochastic optimal control problem %
\begin{equation}\label{eq:problem 01}
{\bf Problem~1:}\quad\quad\min_{u}\mathbb{E}\left[\int_{0}^{t_f}\frac{1}{2}\|u(t)\|^2dt\right],
\end{equation}
subject to
\begin{align}\label{eq:uncertain states 001}
&d X(t,\theta)=(A(\theta)X(t,\theta)+B(\theta)u(t))d t+\sqrt{\epsilon}B(\theta)d W(t),\nonumber\\
&X(0,\theta)\sim\mu_0\quad\text{ and }\int_{0}^{1} X(t_f,\theta)d\theta\sim\mu_f.
\end{align}
Here $X(t,\theta)\in\R^d$ is the random state of an individual system at time $t$ indexed by the sample point $\theta\in\Omega$,  $\map{A}{\Omega}{\mathbb{R}^{d\times d}}$ and $\map{B}{\Omega}{\mathbb{R}^{d\times m}}$ are measurable mappings such that $\|A\|:=\sup_{\theta\in\Omega}\|A(\theta)\|<\infty$ and $\|B\|:=\sup_{\theta\in\Omega}\|B(\theta)\|<\infty$, where the norm here is the Frobenius norm on the space of matrices, $u\in \mathrm{L}^2([0,t_f];\mathbb{R}^m)$ is a \emph{parameter-independent} control input. 
Our problem is motivated by ensemble-based Reinforcement Learning, see~\cite{ADO-CY:23}. 
 
To relate Problem~$1$ to Schr\"{o}dinger bridge~\eqref{eq:original dynamic Schrodinger bridge problem}-\eqref{eq:original Schrodinger bridge problem}, let 
\begin{align}\label{eq: ens_passive_dyn}
&d Y(t,\theta)=A(\theta)Y(t,\theta)d t+\sqrt{\epsilon}B(\theta)d W(t)\nonumber\\
&Y(0,\theta)=x_0
\end{align}
be the passive dynamics and $\mathbf{R}_{y}$ the probability law induced by the averaged process $y(t)=\int_0^1Y(t,\theta)d\theta$ characterized as
\begin{equation}\label{eq: pass_ave_at_t}
y(t)=\left(\int_0^1e^{A(\theta)t}d\theta\right) x_0+ \sqrt{\epsilon}\int_{0}^{t}\Phi(t,\tau) d W(\tau), 
\end{equation}
where 
\begin{equation}\label{eq: non-transition matrix}
\Phi(t_f,\tau) = \int_0^1e^{A(\theta)(t_f-\tau)}B(\theta)d\theta.    
\end{equation}
Then the end-state transition density is characterized as 
\begin{equation}\label{eq: tran_den_for_ensen_passive_dyn}
q^{\epsilon, G}(0,x_0,t_f,x_f)=(2\pi\epsilon)^{-\frac{d}{2}}(\mathrm{det}(G_{t_f,0}))^{-\frac{d}{2}}\exp\left(-\frac{1}{2\epsilon}\left \|x_f-\left(\int_0^1e^{A(\theta)t_f}d\theta\right)x_0\right \|^2_{G_{t_f,0}^{-1}}\right),   
\end{equation}
where $\|x\|^2_{G_{t_f,0}^{-1}}=x^{\mathrm{T}}G_{t_f,0}^{-1}x$, for all $x\in\R^d$, whenever 
\begin{equation}\label{eq: Gramian}
G_{t_f,0}:=\int_{0}^{t_f}\Phi(t_f,\tau)\Phi(t_f,\tau)^{\mathrm{T}}d\tau	   
\end{equation}
is invertible. We find another probability law $\mathbf{P}_{x}$, where $x(t)=\int_0^1X(t,\theta)d\theta$,  whose averaged drift  and marginal constraint satisfy the averaged of~\eqref{eq:uncertain states 001}. To this end, by adding and subtracting $\int_0^t \Phi(t,\tau)u(\tau)d\tau$ of~\eqref{eq: pass_ave_at_t}, we have that 
\begin{equation}\label{eq: ave_output}
 x(t)=\left(\int_0^1e^{A(\theta)t}d\theta\right) x_0+\int_0^t \Phi(t,\tau)(u(\tau)d\tau +\sqrt{\epsilon} d\tilde{W}(\tau)),
\end{equation}
where 
\begin{equation}\label{eq: new_BM}
\tilde{W}(t)=W(t)-\int_0^t\frac{1}{\sqrt{\epsilon}}u(s)ds.   
\end{equation}
Therefore, by Girsanov's Theorem (see~\cite[Chapter~8.6]{OB:03}), we have that the Radon-Nikodym derivative
\begin{equation}\label{eq: Nikodym derivative}
\frac{d \mathbf{P}_{x}}{d \mathbf{R}_{y}}=\exp{\bigg\{\int_0^{t_f}\frac{1}{\sqrt{\epsilon}}u(t)dW(t)+\int_{0}^{t_f}\frac{1}{2\epsilon}\|u(t)\|^2dt\bigg\}},    
\end{equation}
makes~\eqref{eq: new_BM} a Brownian motion. By substituting~\eqref{eq: Nikodym derivative} in~\eqref{eq:original dynamic Schrodinger bridge problem}, we get that
\begin{equation}\label{eq: KL is min_energy}
\infdiv{ \mathbf{P}_{x}}{\mathbf{R}_{y}}=\mathbb{E}\left[\int_{0}^{t_f}\frac{1}{2\epsilon}\|u(t)\|^2dt\right].    
\end{equation}
Hence, problem~\eqref{eq:problem 01}-\eqref{eq:uncertain states 001} aims to find the most likely probability law induced by an \emph{averaged process} that satisfies the end-time marginal. In terms of optimal transport (see~\cite{YC-TTG-MP:21,CL:13,ADO-BT-GB,ADO-GB:24}), problem~\eqref{eq:problem 01}-\eqref{eq:uncertain states 001} describes the deformation from $\mu_0$ to $\mu_f$, where the evolution is affected by both internal and external noise. The internal noise is modeled by parameter perturbation and the external noise is modeled by white noise.

Beyond the fact that in \emph{ all} the related works of Schr\"{o}dinger bridge, one is often concerned with a single Markov process whereas here we deal with an ensemble of Markov processes~\cite{QJ-ZA-LJS:13,BR-KN:00}, complications arise since \emph{we require the control input to be robust or independent of the variations in the parameter $\theta\in[0,1]$ of the system}. Since the object to interest is the averaged or the expectation over the parameter $\theta\in[0,1]$, this adds to the technical challenges of problem~\eqref{eq:problem 01}-\eqref{eq:uncertain states 001}. In fact, \emph{the averaged stochastic process is a non-Markov Gaussian process} (see~\cite{ADO-CY:23}).  This implies that standard stochastic control techniques~\cite{NE:67,VHR:07,KFC:12,CY-GT:15,CY-GTT-PM:15,CY-GTT-PM:16,OB:03,PPD-PM:90, DPP:91} and related references are not readily applicable due to the presence of memory. To overcome this issue, we rely on an approach that is "path-centric" rather than "process-centric". 
This leads us to a functional or path integral approach introduced in~\cite{RPF-FLV:00,RPF-ARH-DFS:10}. Path integral technique for designing optimal Markov controls was first introduced by Kappen~\cite{HJK:05,HJK:2005} and later extended to other Markovian problems~\cite{HJK:07,HJK-WW-DB:07,DB-WW-HJK:08,WW-DB-HJK:08} and reference therein. To the best of our knowledge, there is no precedent for designing non-Markovian control strategies using the path integral formalism. This work is the first. However, conditional distributions for non-Markov process have been computed using path integrals~\cite{AJM-HCL-AJB:90,AJB-AJM-TJN:90,HCL-AJM:90}. 

More precisely, we design the optimal control for~\eqref{eq:problem 01}-\eqref{eq:uncertain states 001}  as a functional integral. 
\emph{We show that the optimal control for~\eqref{eq:problem 01}-\eqref{eq:uncertain states 001} depends on the initial condition and the past and present noise}. This type of control is known as stochastic feedforward control input~\cite{NH:87,MPS:82}  and has applications in flight system control of robotics and crystal growth of power transmission networks (see~\cite{HN-DH-TDB:92,NH:87,MPS:82,HME:89} and reference therein).
In terms of optimal transport framework, our result implies that if the evolution of an initial distribution is an ensemble of stochastic systems, then one can only deform from a given initial distribution to a final distribution along non-Markovian paths using stochastic feedforward controls. Our result unlike in~\cite{CY-GT:15,CY-GTT-PM:15,CY-GTT-PM:16}, relies on the so-called averaged observability inequality~\cite{LM-ZE:14,LJ-ZE:17,LQ-ZE:16,ZE:14} which is equivalent to the invertibility of the matrix~\eqref{eq: Gramian}
(see~\cite{ADO:22}).

The organization of the paper is as follows; To build intuition, we show that the optimal control that steers a single Markov process admits a functional integral approach in Section~\ref{sec: Path Integral Representation of the Optimal Control for a single Markov process}. Then we consider the problem~\eqref{eq:problem 01}-\eqref{eq:uncertain states 001} where $\mu_0$ and $\mu_f$ are Dirac measures at specific points  in Section~\ref{Sec: Problem Statement and Main Result}. This Section aims to characterize optima local controls. The general case, where $\mu_0$ and $\mu_f$ are absolutely continuous measures is discussed in Section~\ref{sec: General Case}. We represent the optimal control for problem~\eqref{eq:problem 01}-\eqref{eq:uncertain states 001} as a functional integral. We conclude with remarks on future work in Section~\ref{sec:Conclusion and future work}.

\section{Path Integral Representation of the Optimal Control for a single Markov process}\label{sec: Path Integral Representation of the Optimal Control for a single Markov process}
To gain an intuition of the characterization of our original problem~\eqref{eq:problem 01}-\eqref{eq:uncertain states 001}, we consider the special case $A(\theta)=A$. In this current case, let $x(t):=\int_0^1X(t,\theta)d\theta$ be the averaged process. Then, 
 we have that~\eqref{eq:problem 01}-\eqref{eq:uncertain states 001} reduces to~\eqref{eq:problem 01} subject to the controlled Markov process 
\begin{align}\label{eq:uncertain states single Markov}
&d x(t)=Ax(t)d t+Bu(t)d t+\sqrt{\epsilon}Bd W(t),\nonumber\\
&x(0)\sim\mu_0\quad\text{and}\quad x(t_f)\sim\mu_f,
\end{align}
where $B:=\int_0^1B(\theta)d\theta$. In this case, the invertibility of the matrix $G_{t_f,0}$ in~\eqref{eq: Gramian} implies that the pair $(A,B)$ is controllable.  This Schrödinger bridge problem is the time-invariant version of the problems in~\cite{CY-GT:15,CY-GTT-PM:15,CY-GTT-PM:16}.  Suppose $d\mu_0=\rho_0dx_0$ and $d\mu_f=\rho_fdx_f$. Then, since the end-point marginal densities $\rho_0$ and $\rho_f$ are fixed, following from~\cite{CY-GTT-PM:2016,CL:13}, we have that that~\eqref{eq:problem 01} subject to~\eqref{eq:uncertain states single Markov} is equivalent to:
\begin{equation}\label{eq:equi–problem 03}
\min_{u}\mathbb{E}\left[\int_{0}^{t_f}\frac{1}{2}\|u(t)\|^2dt-\log\phi_f(x(t_f))\right],
\end{equation}
subject to
\begin{align}\label{eq:uncertain states 01}
&d x(t)=Ax(t)d t+Bu(t)d t+\sqrt{\epsilon}Bd W(t),\nonumber\\
&x(0)=x_0.
\end{align}
Here  $\phi_f$ together with $\phi_0$ solves~\eqref{eq: Schrod_system_cl_c}, where $q(0,x_0,t_f,x_f)=q^{\epsilon, G}(0,x_0,t_f,x_f)$ is the transition density of the passive dynamics
\begin{align}\label{eq: pass_dyn}
&d y(t)=Ay(t)d t+\sqrt{\epsilon}Bd W(t),\nonumber\\
&y(0)=x_0
\end{align}
conditioned that $y(t_f)=x_f$ a.s. This quantity is explicitly characterized as:
\begin{equation}\label{eq: tran_den_for_pass_dyn}
q^{\epsilon, G}(t,x,s,y)=(2\pi\epsilon)^{-\frac{d}{2}}(\mathrm{det}(G_{s,t}))^{-\frac{d}{2}}\exp\left(-\frac{1}{2\epsilon}\left \|y-\left(e^{A(s-t)}\right)x\right \|^2_{G_{s,t}^{-1}}\right),    
\end{equation}
where $G_{s,t}$ defined in~\eqref{eq: Gramian}, with $A(\theta)=A$ and $B(\theta)=B$, is now the controllability Gramian, assumed invertible. We proceed to prove that the optimal control admits a functional integral approach.

\begin{theorem}\label{thm: Integral Form_Markov_Case}
Let $u^*$ be the unique optimal control for~\eqref{eq:equi–problem 03}-\eqref{eq:uncertain states 01}, then $u^*$ admits the path integral formalism 
\begin{equation}\label{eq: class_case}
u^*(t,x)= \int_{\R^d}u^*(t,x|t_f,x_f)p^*(t_f,x_f|t,x)dx_f,
\end{equation}
where $u^*(\cdot|x_f)$ is the unique optimal control for the pinned process, that is~\eqref{eq:uncertain states 01} conditioned that at final time $t=t_f$, we have that $x(t_f)=x_f$ and optimality is measured by
\begin{equation}\label{eq: Mark_obj}
\min_{u}\mathbb{E}\left[\int_{0}^{t_f}\frac{1}{2}\|u(t,x(t))\|^2dt\right]   
\end{equation}
and 
\begin{equation}\label{eq: opt_cont_tran_den}
p^*(t_f,x_f|t,x)=\frac{q^{\epsilon, G}(t,x,t_f,x_f)\phi_f(x_f)}{\int_{\R^d}q^{\epsilon, G}(t,x,t_f,x_f)\phi_f(x_f)dx_f},    
\end{equation}
where $q^{\epsilon, G}$ is in~\eqref{eq: tran_den_for_pass_dyn} and
 $\phi_f$ together with $\phi_0$ solves
\begin{align}\label{eq: Schrödinger system}
\rho_0(x_0)=&\phi_0(x_0)\int_{\R^d}q^{\epsilon, G}(0,x_0,t_f,x_f)\phi_f(x_f)d x_f,\cr
\rho_f(x_f)=&\phi_f(x_f)\int_{\R^d}q^{\epsilon, G}(0,x_0,t_f,x_f)\phi_0(x_0)d x_0,
\end{align}
is the optimal controlled prior distribution.
\end{theorem}
\begin{proof}
Let $\mathbf{P}^*$ be the optimal probability law induced the optimal unpinned process 
\begin{align}\label{eq: opt_uncertain states 01}
&d x(t)=Ax(t)d t+Bu^*(t,x(t))d t+\sqrt{\epsilon}Bd W(t),\nonumber\\
&x(0)=x_0.
\end{align}
and $\rho^*$ be the corresponding optimal controlled density. Consider the disintegration 
\begin{equation}\label{eq: dis_measure}
\mathbf{P}^*(\cdot)=\int_{\R^d}\mathbf{P}^*_{x_f}(\cdot)Q^*(x_f)dx_f.   
\end{equation}
Here, from~\eqref{eq: opt_joint_dis}, we have that $Q^*(x_f)=\phi_0(x_0)q^{\epsilon, G}(0,x_0,t_f,x_f)\phi_f(x_f)$ is the optimal final state density, $\mathbf{P}^*_{x_f}$ is the optimal probability law induced by the optimal pinned process
\begin{align}\label{eq: opt_uncertain states_pinned}
&d x(t)=Ax(t)d t+Bu^*(t,x(t)|t_f,x_f)d t+\sqrt{\epsilon}Bd W(t),\nonumber\\
&x(0)=x_0\quad\text{and}\quad x(t_f)=x_f.
\end{align}
where $u^*(\cdot|t_f,x_f)$, is the optimal conditional control. Here optimality is measured in~\eqref{eq: Mark_obj}. Let $\rho^*(\cdot |t_f,x_f)$ be the optimal conditional density corresponding to~\eqref{eq: opt_uncertain states_pinned}.  Then, the controlled densities $\rho^*$ and $\rho^*(\cdot |t_f,x_f)$ satisfy the Fokker-Planck equations: 
\begin{equation}\label{eq: Fokker-Planck01}
\partial_t \rho^*(t,x)=-\nabla\cdot((x^{\mathrm{T}}A+u^{*\mathrm{T}}(t)B^{\mathrm{T}})\rho^*(t,x))+\frac{\epsilon}{2}\sum_{ij} (BB^{\mathrm{T}})_{ij}\partial^2_{x_i x_j}\rho^*(t,x)    
\end{equation}
and
\begin{equation}\label{eq: Fokker-Planck_cond}
\partial_t \rho^*(t,x|t_f,x_f)=-\nabla\cdot((x^{\mathrm{T}}A+u^{*\mathrm{T}}(t,x|x_f)B^{\mathrm{T}})\rho^*(t,x|t_f,x_f))+\frac{\epsilon}{2}\sum_{ij} (BB^{\mathrm{T}})_{ij}\partial^2_{x_i x_j}\rho^*(t,x|t_f,x_f).    
\end{equation}
Futhermore, from~\eqref{eq: dis_measure} they are coupled through
\[
\rho^*(t,x)=\int_{\R^d}\rho^*(t,x|t_f,x_f)Q^*(x_f)dx_f,
\]
thus
\begin{equation}\label{eq: diff_rela_of_dens}
\partial_t\rho^*(t,x)=\int_{\R^d}\partial_t\rho^*(t,x|t_f,x_f)Q^*(x_f)dx_f.    
\end{equation}
Therefore, from~\eqref{eq: Fokker-Planck01}-\eqref{eq: diff_rela_of_dens}, we have that
\begin{equation*}
u^{*\mathrm{T}}(t)B^{\mathrm{T}}\rho^*(t,x)=\int_{\R^d}u^{*\mathrm{T}}(t|x_f)B^{\mathrm{T}}\rho^*(t,x|t_f,x_f)Q^*(x_f)dx_f.    
\end{equation*}
Thus, 
\begin{equation}\label{eq: con_rela_of_dens}
u^*(t)=\int_{\R^d}u^*(t|x_f)p^*(t_f,x_f|t,x)dx_f,    
\end{equation}
where 
\[
p^*(t_f,x_f|t,x)=\frac{\rho^*(t,x|t_f,x_f)Q^*(x_f)}{\rho^*(t,x)}.
\]
is the optimal controlled prior distribution. Following from~\cite{ET:90}, the optimal controlled prior distribution $p^*(t_f,x_f|\cdot)$ is characterized as~\eqref{eq: opt_cont_tran_den}, where the optimality is measured with respect to the Kullback Leibler ($\mathrm{KL}$) divergence with respect to the uncontrolled prior distribution of~\eqref{eq: pass_dyn}.
\end{proof}

We proceed to verify this result with a well-known result in~\cite{CY-GTT-PM:16}. Following from~\cite{CY-GTT-PM:16} the optimal control for~\eqref{eq:problem 01} subject to~\eqref{eq:uncertain states single Markov} is characterized as the Markov strategy:
\begin{equation}\label{eq: opt_con_for_cl_c}
u^*(t,x)=-\epsilon B^{\mathrm{T}}\nabla\log\varphi (t,x),    
\end{equation}
where
\begin{equation}\label{eq: phi-path}
\varphi(t,x)=\int_{\R^d}q^{\epsilon, G}(t,x,t_f,x_f)\phi_f(x_f)dx_f.   
\end{equation}

Since 
\[
\nabla\log\varphi (t,x)=\frac{\nabla\varphi (t,x)}{\varphi (t,x)},
\]
from~\eqref{eq: opt_con_for_cl_c} and~\eqref{eq: phi-path}, we have that
\begin{equation}\label{eq: opt_con_for_cl_c1}
u(t,x)=-\epsilon B^{\mathrm{T}}\int_{\R^d}\frac{\nabla q^{\epsilon, G}(t,x,t_f,x_f)\phi_f(x_f)}{\int_{\R^d}q^{\epsilon, G}(t,x,t_f,y)\phi_f(y)dy}dx_f.    
\end{equation}
From~\eqref{eq: tran_den_for_pass_dyn}, since 
\begin{align*}
\nabla q^{\epsilon, G}(t,x,t_f,x_f)=&\nabla\left(-\frac{1}{2\epsilon}\left \|x_f-\left(e^{A(t_f-t)}\right)x\right \|^2_{G_{t_f,t}^{-1}} \right)q(t,x,t_f,x_f)\\
=&\frac{1}{\epsilon}e^{A^{\mathrm{T}}(t_f-t)}G_{t_f,t}^{-1}\left(\left(e^{A(t_f-t)}\right)x-x_f\right)q(t,x,t_f,x_f),
\end{align*}
from~\eqref{eq: opt_con_for_cl_c1}, we have that 
\begin{equation}\label{eq: opt_con_for_cl_c2}
u^*(t,x)= \int_{\R^d}u^*(t,x|t_f,x_f)\left(\frac{q^{\epsilon, G}(t,x,t_f,x_f)\phi_f(x_f)}{\int_{\R^d}q^{\epsilon, G}(t,x,t_f,y)\phi_f(y)dy}\right)dx_f,     
\end{equation}
where 
\begin{align}\label{eq: opt_pinn_con}
u^*(t,x|t_f,x_f):=&-B^{\mathrm{T}}e^{A^{\mathrm{T}}(t_f-t)}G_{t_f,t}^{-1}\left(\left(e^{A(t_f-t)}\right)x-x_f\right)\nonumber\\
=&-B^{\mathrm{T}}P(t)^{-1}\left(x-\left(e^{A(t-t_f)}\right)x_f\right).
\end{align}
Here 
\begin{equation}\label{eq: inv_Lyapunov}
P(t)^{-1}:=e^{-A^{\mathrm{T}}(t-t_f)}G_{t_f,t}^{-1}e^{-A(t-t_f)}
\end{equation}
and one can check that $P(t):=e^{A(t-t_f)}G_{t_f,t}e^{A^{\mathrm{T}}(t-t_f)}$ is the Lyapunov function that satisfies
\begin{equation}\label{eq: Lyapunov equation}
\frac{dP(t)}{dt}=AP(t)+P(t)A^T-BB^T,
\end{equation}
subject to $P(t_f)=0$. Therefore, from~\cite[Section~V]{CY-GT:15}), we have that~\eqref{eq: opt_pinn_con}
is the unique optimal control for~\eqref{eq:problem 01} subject to~\eqref{eq:uncertain states single Markov} where $\mu_f=\delta_{x_f}$ is the Dirac measure.  

Therefore, from~\eqref{eq: opt_con_for_cl_c2} and~\eqref{eq: opt_cont_tran_den},  we have that 
\begin{equation}\label{eq: class_case_fun_int}
u^*(t,x)= \int_{\R^d}u^*(t,x|t_f,x_f)p^*(t_f,x_f|t,x)dx_f.    
\end{equation}
This verifies the result.
\begin{remark}\label{rem: Why generalization to non-Markov is true}
Since the optimal control in~\eqref{eq: class_case} is represented as the expectation of optimal local control, where the expectation is over sample paths of uncontrolled process~\eqref{eq: pass_dyn}, the representation in~\eqref{eq: class_case} is "path-centric", rather than the "process-centric" in~\eqref{eq: opt_con_for_cl_c}. The advantage of this "path-centric" representation is that it sidesteps the Markov/non-Markov distinction and represents a more general framework of optimal control. Thus, the path integral method can be generalized to non-Markovian processes~\cite{AJM-HCL-AJB:90,AJB-AJM-TJN:90,HCL-AJM:90}. 
\end{remark}

 \section{Stochastic Bridge of Dirac measures}\label{Sec: Problem Statement and Main Result}
We will represent the optimal control for~\eqref{eq:problem 01}-\eqref{eq:uncertain states 001} as the functional integral of local controls. This Section aims to provide the characterization of optimal local controls. To this end, consider an ensemble of processes governed by~\eqref{eq: ens_passive_dyn}. \emph{Our goal is to find solutions that are conditioned to have} 
\begin{equation}\label{eq:final condition}
\int_{0}^{1} Y(t_f,\theta)d\theta=x_f, a.s.
\end{equation}

The fact that the final conditional state is parameter-independent motivates the quest to find a parameter-independent control.  If the control depends on $\theta\in[0,1]$, it might lead to different behaviours for different realizations, making it challenging to ensure that~\eqref{eq:final condition} is satisfied a.s.  

Following from~\cite{PPD-PM:90,DPP:91}, the solutions of~\eqref{eq: ens_passive_dyn}  conditioned to be~\eqref{eq:final condition} is characterized by the stochastic optimal control problem
\begin{equation}\label{eq:problem 1}
{\bf Problem~2:}\quad\quad\min_{u\in \mathcal{U}}\mathbb{E}\left[\int_{0}^{t_f}\frac{1}{2}\|u(t)\|^2dt\right],
\end{equation}
subject to 
\begin{align}\label{eq:uncertain states}
&d X(t,\theta)=A(\theta)X(t,\theta)d t+B(\theta)u(t)d t+\sqrt{\epsilon}B(\theta)d W(t),\cr
&X(0,\theta)=x_0\text{ a.s and }\int_{0}^{1} X(t_f,\theta)d\theta=x_f, a.s.
\end{align}
To be more precise, if $u\in\mathcal{U}\subset \mathrm{L}^2([0,t_f];\mathbb{R}^m)$, then;
\begin{enumerate}
\item $u(t)$ is $x(t)$-measurable, where $x(t):=\int_{0}^{1} X(t,\theta)d\theta$,  for all $t\in [0,t_f]$,
\item $\mathbb{E}\left[\int_{0}^{t_f}\frac{1}{2}\|u(t)\|^2dt\right]<\infty$,
\end{enumerate}
We proceed to the main theorem of this Section.

\begin{theorem}\label{thm: Main}
Suppose $G_{t_f,s}$, for all $0\leq s<t_f$, in~\eqref{eq: Gramian} is invertible. Then the optimal control for~\eqref{eq:problem 1} subject to~\eqref{eq:uncertain states} is characterized as 
\begin{multline}\label{eq: optimal control}
  u^*(t|x_0,x_f)= -\sqrt{\epsilon}\int_0^t\Phi(t_f,t)^T G_{t_f,\tau}^{-1}\Phi(t_f,\tau) dW(\tau)\\+\Phi(t_f,t)^{T}G_{t_f,0}^{-1} \left(x_f-\left(\int_{0}^{1}e^{A(\theta)t_f}d\theta\right) x_0\right).    
 \end{multline} 
 where $\Phi(t_f,\tau)$ is defined in~\eqref{eq: non-transition matrix}. 
\end{theorem}
We see that the characterization of the local optimal control is a departure from the conventional stochastic optimal control literature, where the optimal control assumes the form of a Markov strategy~\cite{OB:03,VHR:07}.  In particular, when dealing with a non-Markov process subject to parameter perturbations, \emph{the optimal control depends on the initial condition and the entire history of the noise}. For this reason, the optimal control that steers the stochastic bridge adopts an approach—\emph{a stochastic feedforward input}, to be precise. 

\begin{remark}
Before we delve into the proof of Theorem~\eqref{thm: Main}, we must verify the result with the special case $A(\theta)=A$. In this case the problem of requiring the behaviour of 
 \begin{align}\label{eq:linear process}
dy(t)=&Ay(t)dt+\sqrt{\epsilon}Bd W(t),\cr
y(0)=&x_0,\quad\text{almost surely (a.s)},
\end{align}
where $y(t)=\int_0^1Y(t,\theta)d\theta$ is the averaged process and $B=\left(\int_0^1B(\theta)d\theta\right)$, at the final time to be
$x(t_f)=x_f$ a.s is equivalent to optimally changing the drift in
\begin{align}\label{eq: chan_drift_sys}
&d x(t)=Ax(t)d t+Bu(t)d t+\sqrt{\epsilon}Bd W(t),\nonumber\\
&x(0)=x_0\quad\text{and}\quad x(t_f)=x_f\quad \text{a.s},
\end{align} 
where optimality is measured in~\eqref{eq:problem 1}. From~\eqref{eq: opt_pinn_con}-\eqref{eq: inv_Lyapunov}, the optimal pinned local control is characterized as  
\begin{equation}\label{eq: orig_loc_opt}
u^*(t,x(t)|x_f)=-B^{\mathrm{T}}e^{-A^{\mathrm{T}}(t-t_f)}G_{t_f,t}^{-1}e^{-A(t-t_f)}x(t)+B^{\mathrm{T}}e^{-A^{\mathrm{T}}(t-t_f)}G_{t_f,t}^{-1}x_f.    
\end{equation}
By substituting the above optimal pinned control to~\eqref{eq: chan_drift_sys}, we get the optimal pinned process to be
\begin{align}\label{eq: pinned opt_proc}
d x^*(t)=&(A-BB^{\mathrm{T}}e^{-A^{\mathrm{T}}(t-t_f)}G_{t_f,t}^{-1}e^{-A(t-t_f)})x^*(t)dt+B^{\mathrm{T}}e^{-A^{\mathrm{T}}(t-t_f)}G_{t_f,t}^{-1}x_fdt+\sqrt{\epsilon}Bd W(t)\nonumber\\
x(0)=&x_0.
\end{align}
This solution $x^*(t)$ is substituted back in~\eqref{eq: orig_loc_opt} and we demonstrate it coincides with the characterization in~\eqref{eq: optimal control}. Due to the heavy computation, we demonstrate this for $A(\theta)=A=0$. Then from~\eqref{eq: Gramian}, we have that $G_{t_f,\tau}=(t_f-\tau)BB^{\mathrm{T}}$ and from~\eqref{eq: non-transition matrix}, we have that $\Phi(t_f,t)=B\in\R^{d\times d}$, for all $t\geq t_f$ is invertible. Therefore, the optimal local control in~\eqref{eq: orig_loc_opt} and the corresponding optimal state trajectory in~\eqref{eq: pinned opt_proc} reduces to 
\begin{equation}\label{eq: reduced_orig_loc_opt}
u^*(t,x(t)|x_f)=-\frac{B^{-1}}{t_f-t}(x(t)-x_f).   
\end{equation}
and
\begin{align}\label{eq: reduced_pinned opt_proc}
&d x^*(t)=\left(-\frac{x^*(t)}{t_f-t}+\frac{x_f}{t_f-t}\right)dt+\sqrt{\epsilon}Bd W(t),\nonumber\\
&x(0)=x_0\quad a.s
\end{align}
respectively. Since 
\[
x^*(t)=\frac{t_f-t}{t_f}x_0+\frac{t}{t_f}x_f+\sqrt{\epsilon}\int_0^t\frac{t_f-t}{t_f-\tau}BdW(\tau),
\]
by substituting $x^*(t)$ in~\eqref{eq: reduced_orig_loc_opt} we have that
\begin{equation*}
u^*(t|x_0,x_f)=-\frac{B^{-1}}{t_f-t}\bigg(\frac{t_f-t}{t_f}x_0+\frac{t-t_f}{t_f}x_f+\sqrt{\epsilon}\int_0^t\frac{t_f-t}{t_f-\tau}BdW(\tau)\bigg)    
\end{equation*}
and hence
\begin{equation*}
u^*(t|x_0,x_f)=-\sqrt{\epsilon}\int_0^t\frac{1}{t_f-\tau}dW(\tau)+\frac{B^{-1}}{t_f}(x_f-x_0).   
\end{equation*}
This coincides with the proposed characterization~\eqref{eq: optimal control} for the simple case.   
\end{remark}

\subsection{Proof of Theorem~\ref{thm: Main}}
Note that re-centring the initial ensemble $X(0,\theta)$ of systems at the origin $0$ holds no bearing on the system's characterization, given its inherent linearity. For simplicity, from now on we take $X(0,\theta)=0$ a.s.

We outline the proof of Theorem~\ref{thm: Main} as follows: We first state a free-endpoint alternative formulation (cf. \eqref{eq: free end-point problem}–\eqref{eq: output_alter}) and consider the corresponding time discretization (cf. \eqref{eq: discrete-time free end-point problem}–\eqref{eq: discrete-time system}). The characterization in \eqref{eq: optimal control} is derived from the solution of the discretized free-endpoint alternative formulation. The technique of considering an alternative free-endpoint stochastic optimal control is motivated by \cite{CY-GT:15,NE:67,VHR:07,KFC:12}.

Following from~\cite{ADO-CY:23}, since the averaged process in~\eqref{eq: ave_output} is non-Markovian, we circumvent this complexity considering
\begin{equation}\label{eq: free end-point problem}
\min_{u}\mathbb{E}  \bigg[a(x(t_f)-x_f)^T(x(t_f)-x_f) +\int_0^{t_f} \frac{1}{2}\|u_a(t)\|^2 d\tau\bigg] 
\end{equation}
subject to
\begin{equation}\label{eq: output_alter}
 x(t)=\int_0^t \Phi(t,\tau)(u_a(\tau)d\tau +\sqrt{\epsilon} dW(\tau)).
\end{equation}
The main idea of the above formulation is that if  
\[
\lim_{a\rightarrow\infty}\|u^*_{a}-u^*\|^2_{L^2(\mathbb{P})}=0,
\]
holds, where $\mathbb{P}$ is the probability on the sample space $\Omega$. Then $u^*$ is the optimal control for~\eqref{eq:problem 1}-\eqref{eq:uncertain states}. 
The following result is important to design the optimal control $u_a^*$ for~\eqref{eq: free end-point problem}-\eqref{eq: output_alter}.
\begin{proposition}
Consider the discrete-time optimal control problem
\begin{equation}\label{eq: discrete-time free end-point problem}
\min_{u_{a,k}}\mathbb{E}  \bigg[a(x_k-x_f)^T(x_k-x_f) +\frac{1}{2}\sum_{i=0}^{k-1} u_{a,k,i}^Tu_{a,k,i}\triangle t_k\bigg], 
\end{equation}
where $a>0$ and
\begin{equation}\label{eq: discrete-time system}
x_k=\sum_{i=0}^{k-1}\Phi(t_f,t_i)\left(u_{a,k,i}\triangle t_k+\sqrt{\epsilon}\triangle W_{i} \right).
\end{equation}
Here $P:=\{0= t_0<t_1<\dots<t_{k-1}=t_f\}$ be a regular partition with constant step size $\triangle t_k=t_{i+1}-t_{i}$, $\triangle W_{i}:=W(t_{i+1})-W(t_{i})\in\R^m$,  $u_{a,k}:=(u_{a,k,0},\dots,u_{a,k,k-1})\in\R^{mk}$, where we assume that $u_{a,k,i}$ is constant $x_i$-measurable in $L^2[t_i,t_{i+1})$, where $i\in\{0,\dots,k-1\}$. Then the optimal control for~\eqref{eq: discrete-time free end-point problem}-~\eqref{eq: discrete-time system} is characterized as
 \begin{multline}\label{eq: optimal control for free-endpoint problem}
 u_{a,k,i}^*= -\sqrt{\epsilon}\sum_{j=0}^i\Phi(t_f,t_i)^T (\sum_{\alpha=j}^{k-1}\Phi(t_f,t_{\alpha})\Phi(t_f,t_\alpha)^T\triangle t_k+\frac{1}{2a}I_{n})^{-1}\\\Phi(t_f,t_j)\triangle W_j+\Phi(t_f,t_i)^{T}(\sum_{\alpha=0}^{k-1}\Phi(t_f,t_{\alpha})\Phi(t_f,t_\alpha)^T\triangle t_k+\frac{1}{2a}I_{n})^{-1}x_f. 
 \end{multline}       
\end{proposition}

\begin{proof} 
To facilitate the upcoming proof, we'll adopt a convention: unless otherwise specified, we'll represent block matrices using uppercase letters.

We are confronted with a discrete-time problem~\eqref{eq: discrete-time free end-point problem}-\eqref{eq: discrete-time system} which we condense into a more manageable form:
\begin{equation}\label{eq: compact form of the discrete-time free end-point problem}
\min_{u_{a,k}}J_k(u_{a,k}):=\mathbb{E}  \bigg[a\mathrm{Tr}\left((x_k-x_f)(x_k-x_f)^T\right) +\\\frac{\triangle t_k}{2}\mathrm{Tr}\left(u_{a,k}u_{a,k}^T\right)\bigg],    
\end{equation}
subject to 
\begin{equation}\label{eq: compact form of discrete-time system}
x_{k}=\Phi_k\left(u_{a,k}\triangle t_k+\sqrt{\epsilon}\gamma_k \right),
\end{equation}
where $\gamma_k=(\triangle W_{0},\dots,\triangle W_{k-1})\in\R^{mk}$ and $\Phi_k=(\Phi_0(t_f),\dots,\Phi_{k-1}
(t_f))\in\R^{d\times mk}$, where $\Phi_i(t_f):=\Phi(t_f,t_i)\in\R^{d\times m}$ in~\eqref{eq: non-transition matrix}. Since $u_{a,k,i}$ is constant $x_i$-measurable, for all $i=0,\dots,k-1$, it is characterized as 
\begin{equation}\label{eq: opt_chara_dis}
u_{a,k,i}=\sum_{j=0}^iF_{k,i,j}\triangle W_{j}+G_{i,k-1}x_f,    
\end{equation}
where $F_{k,i,j}\in\R^{m\times m}$ and $G_{i,k-1}\in\R^{m\times d}$ are to be determined. Specifically, $F_{k,i,j}:=F(t_i,t_j)\in\R^{m\times m}$, for all $t_j<t_i$, and $G_{i,k-1}:=G(t_i,t_{k-1})\in\R^{m\times d}$. For later use, we express~\eqref{eq: opt_chara_dis} compactly as:
\begin{equation}\label{eq: admissible discrete control}
u_{a,k}=\mathrm{F}_k\gamma_k+\mathrm{G}_k\mathrm{vec}(x_f),    
\end{equation}
where $\mathrm{vec}(x_f)\in\R^{dk}$ is the vector with all its entries $x_f\in\R^d$, the block matrices $\mathrm{F}_k\in\mathcal{T}_{mk}$ and $\mathrm{G}_k\in\mathcal{M}_{mk}$ where 
\begin{equation*}
\mathcal{T}_{mk}:=\{ \mathrm{F}_k\in\R^{mk\times mk} : \chi_{mk}\odot \mathrm{F}_k=0_{mk}\},   
\end{equation*}
and
\begin{equation*}
\mathcal{M}_{mk}:=\{ \mathrm{G}_k\in\R^{mk\times dk} : \Gamma_{mk}\odot \mathrm{G}_k=0_{mk}\}.  
\end{equation*}
Here, $0_{mk}\in\R^{mk\times mk}$ is the zero matrix, $\chi_{mk}$ is the indicator matrix defined as:
\begin{equation}\label{eq: up_tri_0}
\chi_{mk}:=\begin{cases}
\mathbbm{1}_{m}&\text{if }\quad i<j\cr
0_{m} &\text{otherwise},
\end{cases}
\end{equation}
where $\mathbbm{1}_{m}\in\R^{m\times m}$ is the matrix where all entries are $1$s. Similarly, $\Gamma_{mk}$ is defined as:
\begin{equation}\label{eq: last col_0}
\Gamma_{mk}:=\begin{cases}
\mathbbm{1}_{m\times d}&\text{if }\quad j\neq k-1\cr
0_{m\times d} &\text{if }\quad j=k-1
\end{cases}
\end{equation}
with all its entries in the last column set to $0_{m\times d}\in\R^{m\times d}$, which represents the zero matrix. All other entries are assigned the block matrix $\mathbbm{1}_{m\times d}\in\R^{m\times d}$. The operation $\odot$ denotes the Hadamard matrix multiplication operation.
By substituting~\eqref{eq: admissible discrete control} in~\eqref{eq: compact form of discrete-time system}, we arrive at 
\begin{equation}\label{eq: compact form of discrete-time system2}
x_{k}=\Phi_k\left(\mathrm{F}_k\triangle t_k+\sqrt{\epsilon}\mathrm{I}_{mk}\right)\gamma_k+\Phi_k \mathrm{G}_k\mathrm{vec}(x_f)\triangle t_k.
\end{equation}
Here, $\mathrm{I}_{mk}\in\R^{mk\times mk}$ represents the identity matrix. Subsequently, by substituting~\eqref{eq: admissible discrete control} and~\eqref{eq: compact form of discrete-time system2} back into~\eqref{eq: compact form of the discrete-time free end-point problem} and using the fact that $\mathcal{T}_{mk}\cap\mathcal{M}_{mk}=\emptyset$, 
we have two separate strictly convex subproblems 
\begin{multline}\label{eq: mini over LR}
\min_{\mathrm{G}_k\in\mathcal{M}_{mk}} \mathrm{Tr}\bigg(a\bigg(\Phi_k \mathrm{G}_k\mathrm{vec}(x_f)\triangle t_k-x_f\bigg)\bigg(\Phi_k \mathrm{G}_k\mathrm{vec}(x_f)\triangle t_k-x_f\bigg)^T\\+\frac{\triangle t_k}{2}\mathrm{G}_k\mathrm{vec}(x_f)\mathrm{vec}(x_f)^T\mathrm{G}_k^T\bigg)      
\end{multline}
and
\begin{multline}\label{eq: mini over LT}
\min_{\mathrm{F}_k\in\mathcal{T}_{mk}}\mathrm{Tr}\bigg(a\Phi_k\bigg(\mathrm{F}_k\triangle t_k+\sqrt{\epsilon}\mathrm{I}_{mk}\bigg)\bigg(\mathrm{F}_k\triangle t_k+\sqrt{\epsilon}\mathrm{I}_{mk}\bigg)^T\Phi_k^T\triangle t_k+\frac{(\triangle t_k)^2}{2}\mathrm{F}_k\mathrm{F}_k^T\bigg)
\end{multline}
over compact sets. Hence the standard Lagrange method (see~\cite{BSP-VL:04}) gives the unique solution. 

For problem~\eqref{eq: mini over LR}, this necessitates the introduction of a Lagrange multiplier, denoted as $\lambda_k\in\R^{mk\times mk}$, such that the pair $(\mathrm{G}_k,\lambda_k)$ satisfy the first-order conditions: 
\begin{multline*}
\Gamma_{mk}\odot \lambda_k+2a\Phi_k^T\Phi_k\mathrm{G}_k\mathrm{vec}(x_f)\mathrm{vec}(x_f)^T(\triangle t_k)^2-\\2a\Phi_k^Tx_f\mathrm{vec}(x_f)^T\triangle t_k+\mathrm{G}_k\mathrm{vec}(x_f)\mathrm{vec}(x_f)^T\triangle t_k=0_{mk}  
\end{multline*}
and
\[
\Gamma_{mk}\odot \mathrm{G}_k=0_{mk}.
\]
Hence, we have that
\begin{multline}\label{eq: Optimal G equation}
\mathrm{G}^*_k\mathrm{vec}(x_f)\mathrm{vec}(x_f)^T=2a\left(\mathrm{I}_{mk}-\Phi_k^TA_k\Phi_k\triangle t_k\right)\left(\Phi_k^Tx_f\mathrm{vec}(x_f)^T-\Gamma_{mk}\odot\frac{\lambda_k}{2a(\triangle t_k)}\right)    
\end{multline}
and
\begin{equation}\label{eq: Lag_eqn_G}
\Gamma_{mk}\odot\bigg(\left(\mathrm{I}_{mk}-\Phi_k^TA_k\Phi_k\triangle t_k\right)\left(\Phi_k^Tx_f\mathrm{vec}(x_f)^T-\Gamma_{mk}\odot\frac{\lambda_k}{2a(\triangle t_k)}\right)\bigg)=0_{mk}, 
\end{equation}
where
\begin{align*}
 A_k=&\left(\Phi_k\Phi_k^T\triangle t_k+\frac{1}{2a}\mathrm{I}_{dk}\right)^{-1}\cr=&\left(\sum_{\alpha=0}^{k-1}\Phi_{\alpha}(t_f)\Phi_{\alpha}(t_f)^T\triangle t_k+\frac{1}{2a}\mathrm{I}_{dk}\right)^{-1}.  
\end{align*}
By computing $\frac{\lambda_k}{2a(\triangle t_k)}$ in~\eqref{eq: Lag_eqn_G} and substituting in~\eqref{eq: Optimal G equation}, we get that
\begin{equation}\label{eq: last column of optimal G }
G^*_{i,k-1}=\Phi(t_f,t_i)^{T}\bigg(\sum_{\alpha=0}^{k-1}\Phi(t_f,t_{\alpha})\Phi(t_f,t_\alpha)^T\triangle t_k+\frac{1}{2a}I_{n}\bigg)^{-1}.  
\end{equation}

Similarly, for problem~\eqref{eq: mini over LT}, without repeating arguments, one can show that for $i\geq j$ and $j>0$, the following characterization holds:
\begin{equation}\label{eq: final optimal entry for F}
F^*_{k,i,j}=-\sqrt{\epsilon}\Phi(t_f,t_i)^{T}\bigg(\sum_{\alpha=j}^{k-1}\Phi(t_f,t_{\alpha})\Phi(t_f,t_\alpha)^T\triangle t_k+\frac{1}{2a}I_{n}\bigg)^{-1}\Phi(t_f,t_j).
\end{equation}
Therefore, from~\eqref{eq: opt_chara_dis},~\eqref{eq: last column of optimal G } and~\eqref{eq: final optimal entry for F}, we have that the optimal control for~\eqref{eq: discrete-time free end-point problem}-~\eqref{eq: discrete-time system} is~\eqref{eq: optimal control for free-endpoint problem}.
\end{proof}

We now proceed to the proof of Theorem~\ref{thm: Main}.

\begin{proof}[Proof of Theorem~\ref{thm: Main}]
From~\eqref{eq: optimal control for free-endpoint problem},
and~\cite[Definition~3.1.6 and Corollary~3.1.8]{OB:03}, we conclude that 
 \begin{equation*}
\lim_{a\rightarrow\infty}\lim_{k\rightarrow\infty}\| u_{a,k,i}^*-u^*\|^2_{L^2(\mathbb{P})}=0,
 \end{equation*}
 where 
 \begin{equation*}
  u^*(t|x_f)= -\sqrt{\epsilon}\int_0^t\Phi(t_f,t)^T G_{t_f,\tau}^{-1}\Phi(t_f,\tau) dW(\tau)+\Phi(t_f,t)^{T}G_{t_f,0}^{-1}x_f.    
 \end{equation*}
\end{proof}

By translating the averaged process in~\eqref{eq: ave_output}, given any arbitrary $x_0$, the optimal control is characterized as
\begin{multline}\label{eq: opt_loc_con}
  u^*(t|x_0,x_f)= -\sqrt{\epsilon}\int_0^t\Phi(t_f,t)^T G_{t_f,\tau}^{-1}\Phi(t_f,\tau) dW(\tau)\\+\Phi(t_f,t)^{T}G_{t_f,0}^{-1} \left(x_f-\left(\int_{0}^{1}e^{A(\theta)t_f}d\theta\right) x_0\right).    
 \end{multline}
This highlights the fact that optimal local controls depend also on the initial condition $X(0,\theta)=x_0$ and the entire history of noise. 
\section{General Case}\label{sec: General Case}
We are now ready to characterize the optimal control for~\eqref{eq:problem 01} subject to~\eqref{eq:uncertain states 001}.

\begin{theorem}\label{thm: Integral Form_non-Markov_Case} 
Suppose $G_{t_f,t}$, for all $0\leq t<t_f$ defined in~\eqref{eq: Gramian} is invertible. Let $u^*$ be the optimal control for~\eqref{eq:problem 01} subject to~\eqref{eq:uncertain states 001} then 
\begin{equation}\label{eq: gen_class_case}
u^*(t|x_0)= \int_{\R^d}u^*(t|x_0,x_f)p^*(t_f,x_f|0,x_0,t)dx_f,
\end{equation}
where $u^*(t|x_0,x_f)$ is characterized in~\eqref{eq: opt_loc_con} and 
\begin{equation}\label{eq: prior_dis_uncon}
p^*(t_f,x_f|0,x_0,t)=\frac{\hat{q}^{\epsilon, G}(0,x_0,t,t_f,x_f)\phi_f(x_f)}{\int_{\R^d}\hat{q}^{\epsilon, G}(0,x_0,t,t_f,y)\phi_f(y)dy}    
\end{equation}
is the optimal prior distribution where
\begin{multline}\label{eq: fin_tran_den_for_pass_dyn}
\hat{q}^{\epsilon, G}(0,x_0,t,t_f,y)=(2\pi\epsilon)^{-\frac{d}{2}}(\mathrm{det}(G_{t_f,t}))^{-\frac{d}{2}}\\ \exp\bigg(-\frac{1}{2\epsilon}\bigg \|y-\left(\int_0^1e^{A(\theta)t_f}d\theta\right) x_0- \sqrt{\epsilon}\int_{0}^{t}\Phi(t_f,\tau) d W(\tau) \bigg\|^2_{G_{t_f,t}^{-1}}\bigg),    
\end{multline}
is the conditional transition density of~\eqref{eq: pass_ave_at_t} and $\phi_f$ together with $\phi_0$ solves~\eqref{eq: Schrödinger system} with end-state transition density~\eqref{eq: tran_den_for_ensen_passive_dyn}. 
\end{theorem}

Before we provide the proof, note that if $A(\theta)=A$ then, from~\cite{ADO-CY:23}, we have that $y(t)$, where $0\le t\le t_f$, in~\eqref{eq: pass_ave_at_t} is a Markov process. This implies that at time $t=t_f$, we have that $y(t_f)$ only depends on $y(t)$. Therefore, we have that $y(t_f)$ given that $y(t)=x$ is characterized as 
\[
y(t_f)=e^{A(t_f-t)}x+\sqrt{\epsilon}\int_{t}^{t_f}\Phi(t_f,\tau) d W(\tau)
\]
and hence the associated conditional transition density~\eqref{eq: fin_tran_den_for_pass_dyn} reduces to the Markov density in~\eqref{eq: tran_den_for_pass_dyn}. 

We highlight the distinction between the above result and Theorem~\ref{thm: Integral Form_Markov_Case}. While the sample paths in Theorem~\ref{thm: Integral Form_Markov_Case} exhibit Markov property, the sample paths in Theorem~\ref{thm: Integral Form_non-Markov_Case}  depend on the entire history of the noise and the initial condition. Therefore,  the optimal control in~\eqref{eq: gen_class_case} is a stochastic feedforward control input.

\begin{proof}
We first show that the functional integral
\begin{equation}\label{eq: gen_relation_between control and pinned control}
u^*(t,x|x_0)= \int_{\R^d}u^*(t|x_0,x_f)p^*(t_f,x_f|0,x_0,t,x)dx_f,
\end{equation}
holds, where $u^*(t|x_0,x_f)$ is characterized in~\eqref{eq: opt_loc_con} and then provide a characterization for $p^*(t_f,x_f|0,x_0,t,x)$.  Let $u^*$ be the optimal control for~\eqref{eq:problem 01} subject to~\eqref{eq:uncertain states 001}. Then, $u^*$ is the optimal control for the alternative formulation~\eqref{eq:equi–problem 03} subject to 
\begin{align}\label{eq: Alt_ensem_con}
&d X(t,\theta)=A(\theta)X(t,\theta)d t+B(\theta)u(t)d t+\sqrt{\epsilon}B(\theta)d W(t),\nonumber\\
&X(0,\theta)=x_0.
\end{align}
Here $\phi_f$ together with $\phi_0$ solves~\eqref{eq: Schrödinger system}, where $q^{\epsilon, G}(0,x_0,t_f,x_f)$ is the end-state transition density of the passive dynamics~\eqref{eq: ens_passive_dyn} and is characterized in~\eqref{eq: tran_den_for_ensen_passive_dyn}. Let 
\begin{equation}\label{eq: output_with_IC}
 x^*(t|x_0)=\left(\int_0^1e^{A(\theta)t}d\theta\right) x_0+\int_0^t \Phi(t,\tau)(u^*(\tau|x_0)d\tau +\sqrt{\epsilon} dW(\tau)),    
\end{equation}

where  $\Phi(t,\tau)$ is defined in~\eqref{eq: non-transition matrix} be the optimal averaged process of~\eqref{eq: Alt_ensem_con} induced by the optimal control $u^*$. Then, the optimal sample path $ x^*(\cdot|x_0)$ induced an optimal probability law $\mathsf{P}^*(\cdot)$ on $\mathcal{X}$ which we disintegrate into
\begin{equation}\label{eq: dis_measure_non-Markov}
\mathsf{P}^*(\cdot)=\int_{\R^d}\mathsf{P}^*_{x_f}(\cdot)Q^*(x_f)dx_f.   
\end{equation}
 Here, $\mathsf{P}_{x_f}^*(\cdot)$ is the optimal probability law induced by the optimal pinned sample path $x^*(\cdot|x_0,x_f)$, characterized as  
\begin{equation}\label{eq: pinned_output_with_IC}
 x^*(t|x_0,x_f)=\left(\int_0^1e^{A(\theta)t}d\theta\right) x_0+\int_0^t \Phi(t,\tau)(u^*(\tau|x_0,x_f)d\tau +\sqrt{\epsilon} dW(\tau)),
\end{equation}
for all $t\in[0,t_f]$, with $u^*(t|x_0,x_f)$ in~\eqref{eq: opt_loc_con}. From~\eqref{eq: opt_joint_dis}, we have that\\ $Q^*(x_f)=\phi_0(x_0)q^{\epsilon, G}(0,x_0,t_f,x_f)\phi_f(x_f)$ is the optimal final state density with $q^{\epsilon, G}(0,x_0,t_f,x_f)$ in~\eqref{eq: tran_den_for_ensen_passive_dyn}.

 Let $\rho^*(\cdot|x_0)$ and  $\rho(\cdot|x_0,x_f)$ be their corresponding optimal controlled density of the processes~\eqref{eq: output_with_IC} and~\eqref{eq: pinned_output_with_IC}, respectively. Since both processes $x^*(t|x_0)$ and $x^*(t|x_0,x_f)$ in~\eqref{eq: output_with_IC} and~\eqref{eq: pinned_output_with_IC}, respectively, are Gaussian processes with means  
\[
\left(\int_0^1e^{A(\theta)t}d\theta\right) x_0+\mathbb{E}\int_0^t \Phi(t,\tau)u^*(\tau|x_0)d\tau
\]
and
\[
\left(\int_0^1e^{A(\theta)t}d\theta\right) x_0+\mathbb{E}\int_0^t \Phi(t,\tau)u^*(\tau|x_0,x_f)d\tau
\]
respectively, and covariance matrix $\epsilon G_{t,0}$, we have that 
\begin{multline}\label{eq: den_unpin_ave_proc}
\rho^*(t,x|x_0)=(2\pi\epsilon)^{-\frac{d}{2}}(\mathrm{det}(G_{t,0}))^{-\frac{d}{2}}\\ \exp\bigg(-\frac{1}{2\epsilon}\bigg\|x-\left(\int_0^1e^{A(\theta)t}d\theta\right) x_0+\mathbb{E}\int_0^t \Phi(t,\tau)u^*(\tau|x_0)d\tau\bigg \|^2_{G_{t,0}^{-1}}\bigg)     
\end{multline}
and
\begin{multline}\label{eq: den_pin_ave_proc}
\rho^*(t,x|x_0,x_f)=(2\pi\epsilon)^{-\frac{d}{2}}(\mathrm{det}(G_{t,0}))^{-\frac{d}{2}}\\ \exp\bigg(-\frac{1}{2\epsilon}\bigg\|x-\left(\int_0^1e^{A(\theta)t}d\theta\right) x_0+\mathbb{E}\int_0^t \Phi(t,\tau)u^*(\tau|x_0,x_f)d\tau\bigg \|^2_{G_{t,0}^{-1}}\bigg).     
\end{multline}
From~\eqref{eq: dis_measure_non-Markov}, since the densities are coupled through
\begin{equation}\label{eq: coupled_den_of_ave_pro}
\rho^*(t,x|x_0)=\int_{\R^d}\rho^*(t,x|x_0,x_f)Q^*(x_f)dx_f,    
\end{equation}
for all $t,x$, we have that
\begin{equation*}
\nabla\rho^*(t,x|x_0)=\int_{\R^d}\nabla\rho^*(t,x|x_0,x_f)Q^*(x_f)dx_f    
\end{equation*}
hold, for all $t,x$. From~\eqref{eq: den_unpin_ave_proc} and~\eqref{eq: den_pin_ave_proc} the above equation reduces to 
\begin{equation}\label{eq: coupled_derivative}
 m(x,t|x_0)\rho^*(t,x|x_0)=\int_{\R^d}m(x,t|x_0,x_f)\rho^*(t,x|x_0,x_f)Q^*(x_f)dx_f,   
\end{equation}
where
\begin{equation*}
m(x,t|x_0):=x-\left(\int_0^1e^{A(\theta)t}d\theta\right) x_0+\mathbb{E}\int_0^t \Phi(t,\tau)u^*(\tau|x_0)d\tau    
\end{equation*}
and
\begin{equation*}
m(x,t|x_0,x_f):=x-\left(\int_0^1e^{A(\theta)t}d\theta\right) x_0\\+\mathbb{E}\int_0^t \Phi(t,\tau)u^*(\tau|x_0,x_f)d\tau.    
\end{equation*}
Using~\eqref{eq: coupled_den_of_ave_pro}, we have that~\eqref{eq: coupled_derivative} reduces to
\begin{equation*}
\mathbb{E}\int_0^t \Phi(t,\tau)u^*(\tau|x_0)d\tau\rho^*(t,x|x_0)=\int_{\R^d}\bigg(\mathbb{E}\int_0^t \Phi(t,\tau)u^*(\tau|x_0,x_f)d\tau\bigg)\rho^*(t,x|x_0,x_f)Q^*(x_f)dx_f
\end{equation*}
and hence
\begin{equation*}
\int_0^t \Phi(t,\tau)\mathbb{E}\bigg(u^*(\tau|x_0)-\int_{\R^d}u^*(\tau|x_0,x_f)\frac{\rho^*(t,x|x_0,x_f)Q^*(x_f)}{\rho^*(t,x|x_0)}dx_f\bigg)d\tau=0
\end{equation*}
holds, for all $t,x$. Therefore, we have that 
\begin{equation}\label{eq: equal_expect}
\mathbb{E}\bigg(u^*(t,x|x_0)-\int_{\R^d}u^*(t,|x_0,x_f)\frac{\rho^*(t,x|x_0,x_f)Q^*(x_f)}{\rho^*(t,x|x_0)}dx_f\bigg)=0.
\end{equation}
Suppose that
\[
u^*(t,x|x_0)\neq\int_{\R^d}u^*(t,|x_0,x_f)p^*(t_f,x_f|0,x_0,t,x)dx_f,
\]
where 
\begin{equation}\label{eq: opt_prior_dis}
p^*(t_f,x_f|0,x_0,t,x)=\frac{\rho^*(t,x|x_0,x_f)Q^*(x_f)}{\rho^*(t,x|x_0)}    
\end{equation}
is the optimal prior distribution. Then, from~\eqref{eq: equal_expect}, we have that 
\begin{equation}\label{eq: process_difference}
u^*(t,x|x_0)=\int_{\R^d}u^*(t,|x_0,x_f)p^*(t_f,x_f|0,x_0,t,x)dx_f+Z(t)
\end{equation}
holds, where the process $Z(t)$ has $\mathbb{E}(Z(t))=0$. By substituting~\eqref{eq: process_difference} in~\eqref{eq: output_with_IC} we get that
\begin{multline}\label{eq: output_with_IC_with_ext}
 x^*(t|x_0)=\left(\int_0^1e^{A(\theta)t}d\theta\right) x_0+\sqrt{\epsilon}\int_0^t \Phi(t,\tau)dW(\tau)\\+\int_{\R^d}\bigg(\int_0^t \Phi(t,\tau)u^*(\tau|x_0,x_f)p^*(t_f,x_f|0,x_0,\tau,x)d\tau\bigg)dx_f +\int_0^t \Phi(t,\tau)Z(\tau)d\tau.  
\end{multline}
By comparing~\eqref{eq: pinned_output_with_IC} to~\eqref{eq: output_with_IC_with_ext}, the extra sample path $\int_0^{\cdot} \Phi(\cdot,\tau)Z(\tau)d\tau$ implies that $\mathsf{P}^*$ assumed in~\eqref{eq: dis_measure_non-Markov} is not the optimal probability law induced by~\eqref{eq: output_with_IC} which is a contradiction. Therefore, we have that 
\[
u^*(t,x|x_0)=\int_{\R^d}u^*(t,|x_0,x_f)p^*(t_f,x_f|0,x_0,t,x)dx_f
\]
holds.

We proceed to characterize the optimal controlled prior distribution $p^*(t_f,x_f|0,x_0)$.  From~\eqref{eq: pass_ave_at_t}, since 
\begin{align}\label{eq: pass_pro_at fin_time}
y(t_f)=&\left(\int_0^1e^{A(\theta)t_f}d\theta\right) x_0+ \sqrt{\epsilon}\int_{0}^{t_f}\Phi(t_f,\tau) d W(\tau),\\
=&\left(\int_0^1e^{A(\theta)t_f}d\theta\right) x_0+\sqrt{\epsilon}\int_{0}^{t}\Phi(t_f,\tau) d W(\tau)+\sqrt{\epsilon}\int_{t}^{t_f}\Phi(t_f,\tau) d W(\tau),\nonumber
\end{align}
given $x_0$ and $W(\tau)$, where $\,0\le \tau \le t$, we have that $y(t_f)$ follows a Gaussian distribution with mean 
\[
\left(\int_0^1e^{A(\theta)t_f}d\theta\right) x_0+ \sqrt{\epsilon}\int_{0}^{t}\Phi(t_f,\tau) d W(\tau)
\]
and covariance $\epsilon G_{t_f,t}$, where $G_{t_f,t}$ is defined in~\eqref{eq: Gramian}.
Hence, the conditional transition density for the passive averaged process $y(t)$ in~\eqref{eq: pass_ave_at_t} is
\begin{multline}\label{eq: gen_cond_trans_den}
q^{\epsilon, G}(0,x_0,t,t_f,y)=(2\pi\epsilon)^{-\frac{d}{2}}(\mathrm{det}(G_{t_f,t}))^{-\frac{d}{2}}\\ \exp\bigg(-\frac{1}{2\epsilon}\bigg \|y-\left(\int_0^1e^{A(\theta)t_f}d\theta\right) x_0- \sqrt{\epsilon}\int_{0}^{t}\Phi(t_f,\tau) d W(\tau) \|^2_{G_{t_f,t}^{-1}}\bigg),    
\end{multline}
This completes the proof.
\end{proof}
We present a systematic approach to deriving the optimal control for~\eqref{eq:problem 01}-\eqref{eq:uncertain states 001}, as characterized in~\eqref{eq: gen_class_case} in Algorithm~$1$ in $5$ steps. The optimal control in Step~$5$ is specifically designed to effectively transition the averaged process corresponding to~\eqref{eq:uncertain states 001}, described in~\eqref{eq: output_with_IC}, from an initial distribution $\mu_0$ with density $\rho_0$ to a target distribution $\mu_f$ with density $\rho_f$ (see Figure~\ref{fig: Monte Carlo Simulation of target density}). 


\begin{algorithm}
  \caption{Algorithm for obtaining optimal control}
  \begin{algorithmic}
    \Statex \textbullet~\textbf{Given:} 
    \State The ensemble of system
    \[
    d X(t,\theta)=(A(\theta)X(t,\theta)+B(\theta)u(t))d t+\sqrt{\epsilon}B(\theta)d W(t)
    \]
    \State Marginal densities
    \[
    d\mu_0=\rho_0 dx_0\quad\text{and}\quad d\mu_f=\rho_f dx_f.
    \]
    \State Cost
    \[
    \min_{u}\mathbb{E}\left[\int_{0}^{t_f}\frac{1}{2}\|u(t)\|^2dt\right]
    \]
    \Statex \textbullet~\textbf{Use $A(\theta)$ and $B(\theta)$ to compute relevant functions:}
    \begin{enumerate}
        \State  \begin{equation*}
        \Phi(t_f,\tau) =\int_0^1e^{A(\theta)(t_f-\tau)}B(\theta)d\theta\quad\text{and}\quad
        G_{t_f,0}=\int_{0}^{t_f}\Phi(t_f,\tau)\Phi(t_f,\tau)^{\mathrm{T}}d\tau
    \end{equation*}
    \State  \begin{equation*}
    q^{\epsilon, G}(0,x_0,t_f,x_f)=(2\pi\epsilon)^{-\frac{d}{2}}(\mathrm{det}(G_{t_f,0}))^{-\frac{d}{2}}\exp\left(-\frac{1}{2\epsilon}\left \|x_f-\left(\int_0^1e^{A(\theta)t_f}d\theta\right)x_0\right \|^2_{G_{t_f,0}^{-1}}\right)
    \end{equation*}
    \State  \begin{multline*}
\hat{q}^{\epsilon, G}(0,x_0,t,t_f,y)=(2\pi\epsilon)^{-\frac{d}{2}}(\mathrm{det}(G_{t_f,t}))^{-\frac{d}{2}}\\ \exp\bigg(-\frac{1}{2\epsilon}\bigg \|y-\left(\int_0^1e^{A(\theta)t_f}d\theta\right) x_0- \sqrt{\epsilon}\int_{0}^{t}\Phi(t_f,\tau) d W(\tau) \|^2_{G_{t_f,t}^{-1}}\bigg)    
\end{multline*}
    \end{enumerate}

    \Statex \textbullet~\textbf{Use~ $1$) and $A(\theta)$ to compute local control:}
    \State \begin{multline*}
  u^*(t|x_0,x_f)= -\sqrt{\epsilon}\int_0^t\Phi(t_f,t)^T G_{t_f,\tau}^{-1}\Phi(t_f,\tau) dW(\tau)\\+\Phi(t_f,t)^{T}G_{t_f,0}^{-1} \left(x_f-\left(\int_{0}^{1}e^{A(\theta)t_f}d\theta\right) x_0\right)    
 \end{multline*}
    \Statex \textbullet~\textbf{Use~ $3$) to compute  posterior transition densities:} 
    \[
p^*(t_f,x_f|0,x_0,t)=\frac{\hat{q}^{\epsilon, G}(0,x_0,t,t_f,x_f)\phi_f(x_f)}{\int_{\R^d}\hat{q}^{\epsilon, G}(0,x_0,t,t_f,y)\phi_f(y)dy}    
    \]
    where $\phi_f$ together with $\phi_0$ is obtained from
    \begin{align*}
\rho_0(x_0)=&\phi_0(x_0)\int_{\R^d}q^{\epsilon, G}(0,x_0,t_f,x_f)\phi_f(x_f)d x_f\cr
\rho_f(x_f)=&\phi_f(x_f)\int_{\R^d}q^{\epsilon, G}(0,x_0,t_f,x_f)\phi_0(x_0)d x_0
\end{align*}
    \Statex \textbullet~\textbf{Compute optimal control:}
    \State\[
    u^*(t|x_0)= \int_{\R^d}u^*(t|x_0,x_f)p^*(t_f,x_f|0,x_0,t)dx_f
    \]
  \end{algorithmic}
\end{algorithm}

It is important to note that, in step~$4$ since the prior transition density $p^*(t_f,x_f|0,x_0,t)$ accounts for the most likely path from $x_0$ to $x_f$, the integration over all possible final states $x_f$ in step $5$ effectively sums over all potential paths that the passive process in~\eqref{eq: pass_ave_at_t} can take from $x_0$ to the final states, exemplifying the principles of path integral~\cite{HJK:05,HJK:2005,HJK:07,HJK-WW-DB:07,DB-WW-HJK:08,WW-DB-HJK:08}. As noted by the transition density $\hat{q}^{\epsilon, G}$ in Step~$2$, the potential paths that the passive process in~\eqref{eq: pass_ave_at_t} can take depend not only on the initial condition $x_0$  but also on the past noise $\{W(\tau)\}_{0\leq \tau\leq t}$, thereby incorporating the entire historical trajectory at time $t$. This underlines the deviation from~\cite{HJK:05,HJK:2005,HJK:07,HJK-WW-DB:07,DB-WW-HJK:08,WW-DB-HJK:08} and \emph{all} the Schr\"{o}dinger bridge frameworks (see for instance~\cite{CY-GT:15,CY-GTT-PM:15,CY-GTT-PM:16,PPD-PM:90, DPP:91} and reference therein).



\section{Numerical Result}
\begin{example}
We consider one-dimensional passive dynamics
\begin{equation}\label{eq: parameter02}
 d X(t,\theta)=-\theta X(t,\theta)d t+\sqrt{\epsilon}d W(t)   
\end{equation}
over the time horizon $[0,1]$, where $\theta\in[0,1]$. At time $t=0$, the observed initial density of~\eqref{eq: parameter02} is
\begin{equation}\label{eq: onedim_initial density}
\rho_0(x) := \begin{cases} 0.2- 0.2 \cos\left(3\pi x\right)+0.2 & \text{for } 0 \leq x < \frac{2}{3} \\ 5-5 \cos\left(6\pi x-4\pi\right)+0.2 & \text{for } \frac{2}{3} \leq x \leq 1 \end{cases}   
\end{equation}
and at time $t=1$, the observed (or desired) density is 
\begin{equation}\label{eq: onedim_final density}
\rho_f(x):=\rho_0(1-x)    
\end{equation} 
illustrated in Figure~\ref{fig_densities}. This captures a scenario where the transition from~\eqref{eq: onedim_initial density} to~\eqref{eq: onedim_final density} is characterized by the harmonic oscillator, with unknown frequency $\theta$ but only the range  $[0,1]$ is known and the oscillator is influenced by thermal or other external noise with intensity $\epsilon>0$. 
 By solving ~\eqref{eq:problem 01} subject to
\begin{equation}\label{eq: parameter2}
 d X(t,\theta)=(-\theta X(t,\theta)+u(t))d t+\sqrt{\epsilon}d W(t)   
\end{equation} 
the goal is to design control input $u(t)$ that will effectively bridge~\eqref{eq: onedim_final density} from~\eqref{eq: onedim_initial density}, despite the stochastic nature of the process. More precisely,  
the Schr\"{o}dinger bridge problem~\eqref{eq:problem 01} subject to~\eqref{eq: parameter2}, seeks to identify the most likely and robust stochastic harmonic oscillator whose density  $\rho$ have started from~\eqref{eq: onedim_initial density} and has ended at~\eqref{eq: onedim_final density}.

Using the parameters $A(\theta)=-\theta$ and $ B(\theta)=1$ from~\eqref{eq: parameter2}, we have that $G_{t_f,\cdot}$ in~\eqref{eq: Gramian} is computed to be:
\begin{align*}
G_{1,\tau}=\left(\int_{\tau}^1\Phi(1,t)^2dt\right),
\end{align*}
where
\begin{equation}\label{eq: com_non_trans2}
\Phi(1,t)=\frac{1-e^{-(1-t)}}{1-t}.  
\end{equation}
One can check that $\tau\mapsto\int_{\tau}^1\left(\frac{1-e^{-(1-t)}}{1-t}\right)^2 dt$ is strictly monotone on $[0,1]$ and hence 
\begin{equation}\label{eq: com_inv_Gram2}
G^{-1}_{1,\tau}=\left(\int_{\tau}^1\left(\frac{1-e^{-(1-t)}}{1-t}\right)^2 dt\right)^{-1}
\end{equation}
is well-defined for all $\tau\in[0,1)$.  Therefore, from Theorem~\eqref{thm: Integral Form_non-Markov_Case}, we have that the optimal control for problem~\eqref{eq:problem 01}-\eqref{eq:uncertain states 001} where~\eqref{eq: parameter2} is~\eqref{eq: gen_class_case}, demonstrated in Figure~\ref{fig0001} where  $u^*(t|x_0,x_f)$ in~\eqref{eq: opt_loc_con} is demonstrated in Figure~\ref{fig_onedimpin_con} and the conditioned transition density in~\eqref{eq: fin_tran_den_for_pass_dyn} are~\eqref{eq: com_non_trans2}-\eqref{eq: com_inv_Gram2} and
 \begin{equation}\label{eq: ondimint_exp_at_ft}
\int_{0}^{1}e^{-\theta}d\theta=1-e.
\end{equation}    
The corresponding pinned bridge is demonstrated in Figure~\ref{fig_onedimpin_sta} and the most likely and robust stochastic harmonic oscillator is shown in Figure~\ref{fig: optimal averaged process}. From Figure~\ref{fig: Monte Carlo Simulation of target density}, we see that the control in Figure~\ref{fig0001} effectively transitions from $\rho_0$ in~\eqref{eq: onedim_initial density} to $\rho_f$ in~\eqref{eq: onedim_final density} through the most robust likely path in~\eqref{eq: output_with_IC} demonstrated in Figure~\ref{fig: optimal averaged process}.

\end{example}
\begin{example}\label{ex: Example1 }
For the two-dimensional case, we have
\begin{equation}\label{eq: parameter}
A(\theta)=\begin{bmatrix}
    0&-\theta\\
    \theta&0
\end{bmatrix}\quad B(\theta)=\begin{bmatrix}
    1&0\\
    0&1
\end{bmatrix}    
\end{equation}
and hence
\begin{equation*}
e^{A(\theta)(1-t)}=\begin{bmatrix}
\cos(\theta(1-t))&-\sin(\theta(1-t))\\
\sin(\theta(1-t))&\cos(\theta(1-t))
\end{bmatrix}.    
\end{equation*}
Thus, $\Phi(1,\cdot)$ in~\eqref{eq: non-transition matrix} is computed to be:
\begin{equation}\label{eq: com_non_trans}
\Phi(1,t)=\begin{bmatrix}
\frac{\sin(1-t)}{(1-t)}&\frac{\cos(1-t)-1}{(1-t)}\\
\frac{1-\cos(1-t)}{(1-t)}&\frac{\sin(1-t)}{(1-t)}
\end{bmatrix}    
\end{equation}
and $G_{t_f,\cdot}$ in~\eqref{eq: Gramian} is computed to be:
\begin{align*}
G_{1,\tau}=\left(\int_{\tau}^1\Phi(1,t)\Phi(1,t)^{\mathrm{T}}dt\right),
\end{align*}
where
\[
\Phi(1,t)\Phi(1,t)^{\mathrm{T}}=\begin{bmatrix}
\frac{2-2\cos(1-t)}{(1-t)^2}&0\\
0&\frac{2-2\cos(1-\tau)}{(1-t)^2}
\end{bmatrix}. 
\]
Since $\lim_{t\rightarrow 1}\frac{2-2\cos(1-t)}{(1-t)^2}=1$ and the function  $t\mapsto\frac{2-2\cos(1-t)}{(1-t)^2}$ is continuous and strictly monotone on $[\tau,1]$, for all $\tau\in[0,1]$, we have that $\tau\mapsto \int_{\tau}^1\frac{2-2\cos(1-t)}{(1-t)^2}dt$ is invertible on $[0,1)$ and hence
\begin{align}\label{eq: com_inv_Gram}
G^{-1}_{1,\tau}=&\left(\int_{\tau}^{1}\Phi(1,t)\Phi(1,t)^{\mathrm{T}}dt\right)^{-1}\nonumber\\=&\left(\int_{\tau}^{1}\frac{2-2\cos(1-t)}{(1-t)^2}dt\right)^{-1}\begin{bmatrix}
1&0\\
0&1
\end{bmatrix}    
\end{align}
is well-defined for all $\tau\in[0,1)$. Therefore, from Theorem~\eqref{thm: Integral Form_non-Markov_Case}, we have that the optimal control for problem~\eqref{eq:problem 01}-\eqref{eq:uncertain states 001} where~\eqref{eq: parameter} is~\eqref{eq: gen_class_case} where $u^*(t|x_0,x_f)$ in~\eqref{eq: opt_loc_con} is demonstrated in Figure~\ref{fig01} having the conditioned transition density in~\eqref{eq: fin_tran_den_for_pass_dyn} are~\eqref{eq: com_non_trans}-\eqref{eq: com_inv_Gram} and
 \begin{equation}\label{eq: int_exp_at_ft}
\int_{0}^{1}e^{A(\theta)}d\theta=\begin{bmatrix}
\sin(1)&\cos(1)-1\\
1-\cos(1)&\sin(1)
\end{bmatrix}.     
\end{equation} 
The most likely pinned bridge is shown in Figure~\ref{fig02}.
\end{example}

\begin{figure}[!t]
\centerline{\includegraphics[width=\columnwidth]{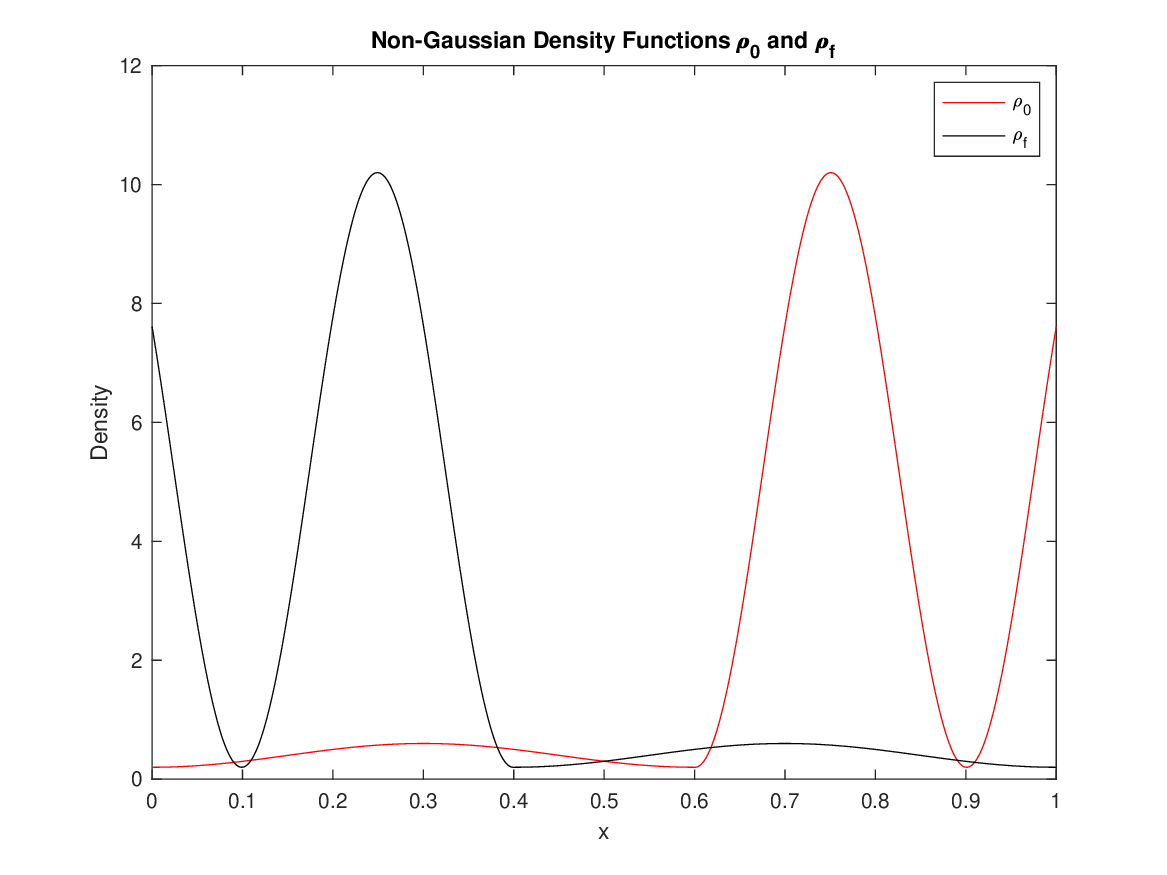}}
\caption{This is the plot of initial and final marginal non-Gaussian densities $\rho_0$ and $\rho_f$ defined in~\eqref{eq: onedim_initial density}-~\eqref{eq: onedim_final density}. This is useful for computing $\phi_0$ and $\phi_f$ in~\eqref{eq: Schrödinger system}.
 }
\label{fig_densities}
\end{figure}

\begin{figure}[!t]
\centerline{\includegraphics[width=\columnwidth]{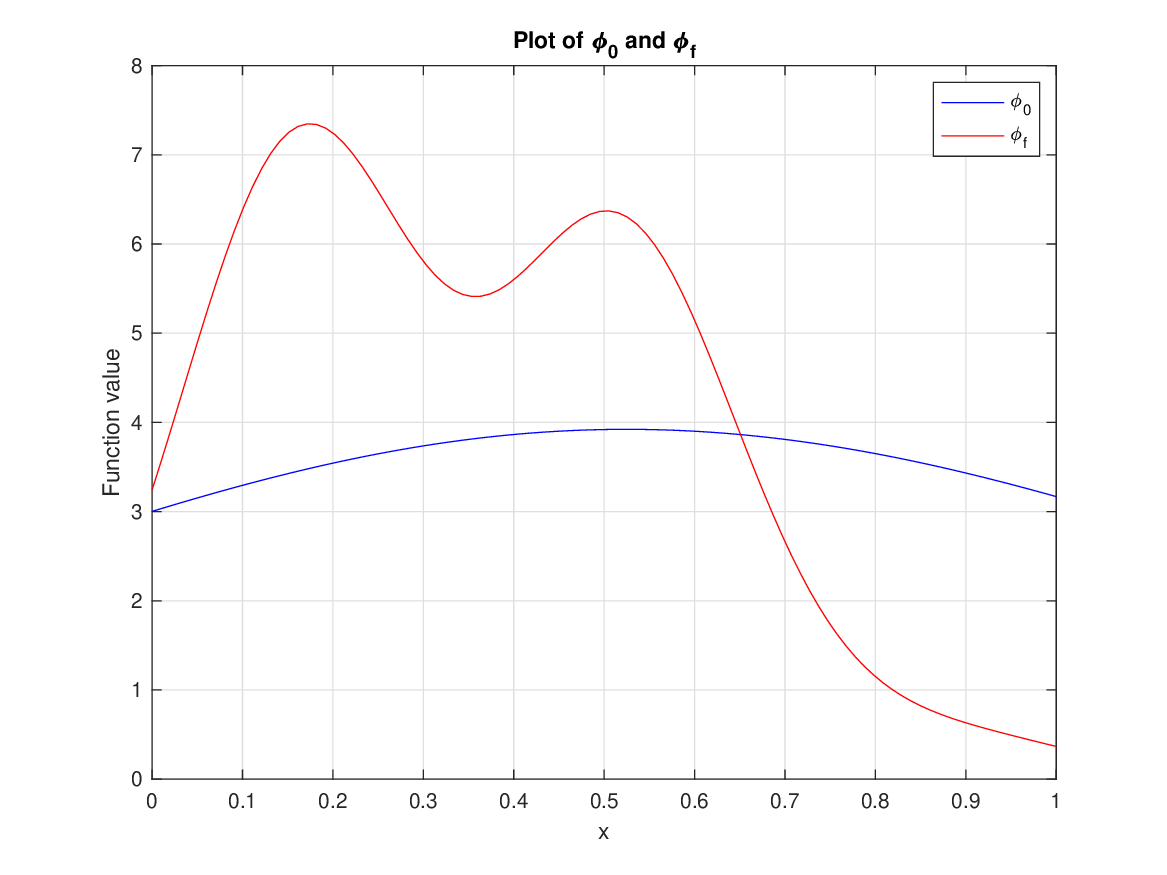}}
\caption{This is the plot of $\phi_0$ and $\phi_f$ computed from~\eqref{eq: Schrödinger system}, where $\rho_0$ and $\rho_f$ are given in~\eqref{eq: onedim_initial density}-\eqref{eq: onedim_final density} with Figure~\ref{fig_densities} and the end-state density function in~\eqref{eq: tran_den_for_ensen_passive_dyn} use the parameters computed in~\eqref{eq: parameter2}-\eqref{eq: com_inv_Gram2}. The function $\phi_f$ (in red) is useful in computing the optimal control in~\eqref{eq: gen_class_case}.
 }
\label{fig_phi_0phif}
\end{figure}

\begin{figure}[!t]
\centerline{\includegraphics[width=\columnwidth]{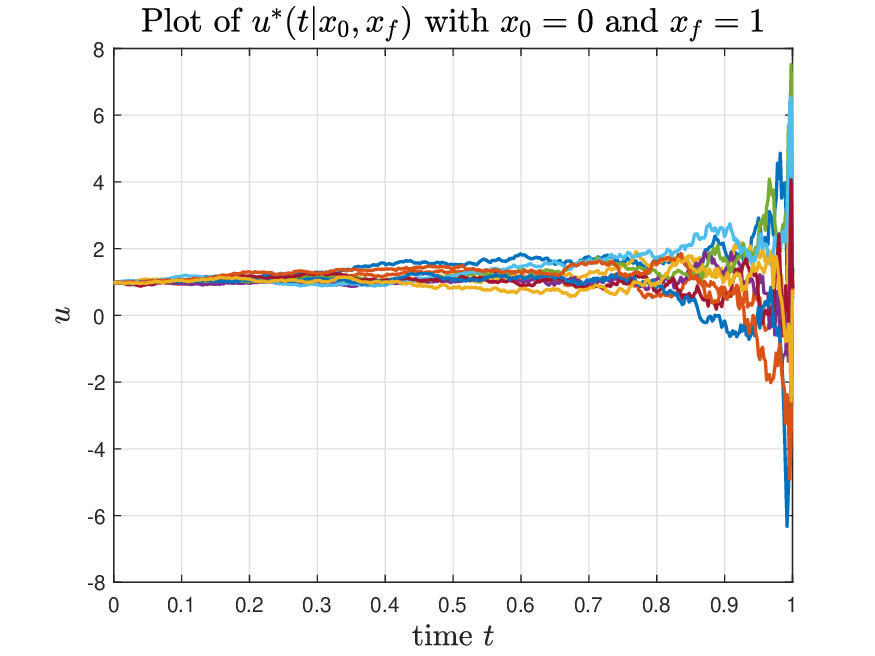}}
\caption{This is the plot of 10 sample paths characterizing the one-dimensional optimal local control process $u^*(t|0,1)$ in~\eqref{eq: opt_loc_con} with parameters in~\eqref{eq: parameter2}-\eqref{eq: ondimint_exp_at_ft}.
 }
\label{fig_onedimpin_con}
\end{figure}

\begin{figure}[!t]
\centerline{\includegraphics[width=\columnwidth]{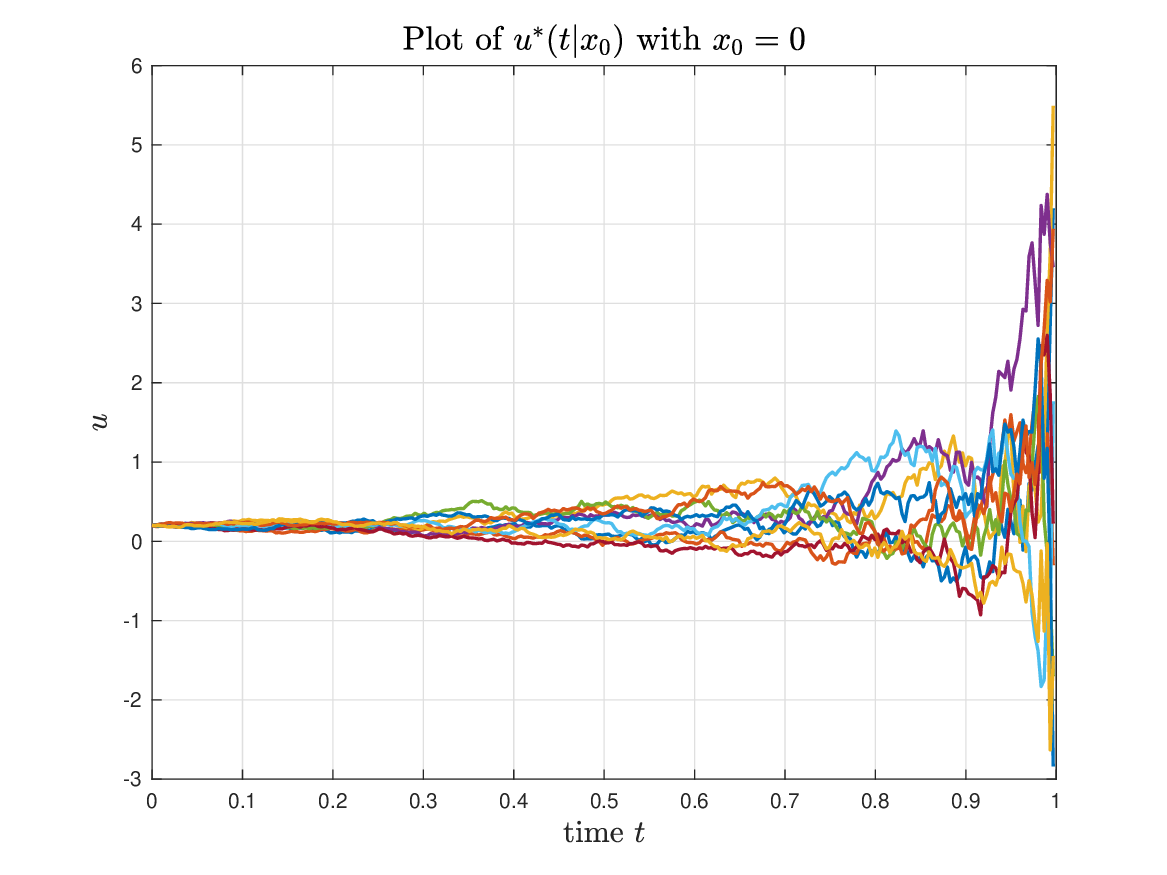}}
\caption{This is the plot of 10 sample paths characterizing the one-dimensional optimal control process $u^*(t|0)$ in~\eqref{eq: gen_class_case} using the pinned control process in Figure~\ref{fig_onedimpin_con} and the prior density in~\eqref{eq: prior_dis_uncon} with parameters in~\eqref{eq: parameter2}-\eqref{eq: ondimint_exp_at_ft},  $\phi_f$ in Figure~\ref{fig_phi_0phif} and $\rho_0$ and $\rho_f$ in Figure~\ref{fig_densities}.
 }
\label{fig0001}
\end{figure} 

\begin{figure}[!t]
\centerline{\includegraphics[width=\columnwidth]{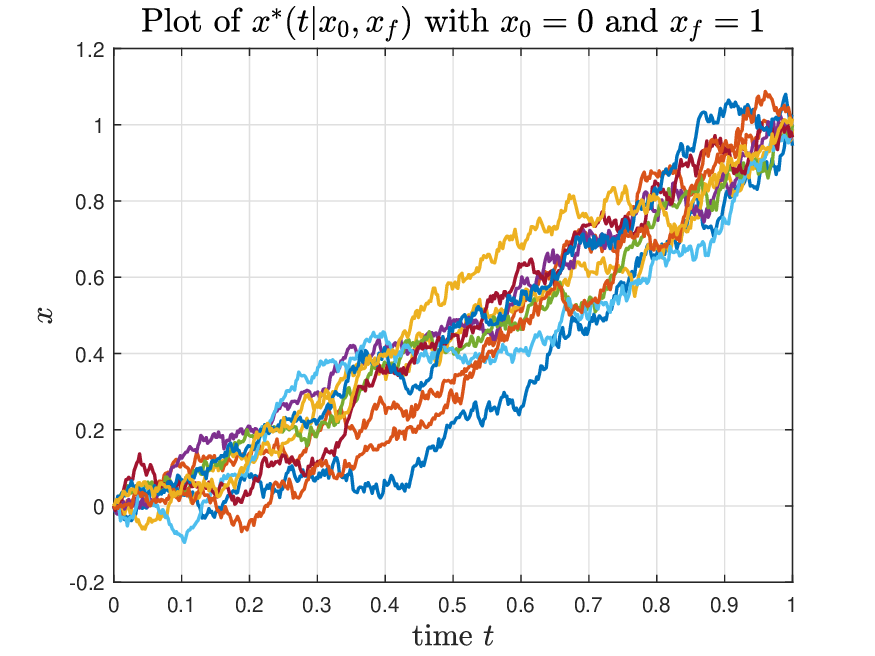}}
\caption{This is the plot of 10 sample paths characterizing the one-dimensional optimal averaged pinned state process $ x^*(t|0,1)$ in~\eqref{eq: pinned_output_with_IC} induced by the optimal local control process in Figure~\ref{fig_onedimpin_con}.} 
\label{fig_onedimpin_sta}
\end{figure} 

\begin{figure}[!t]
\centerline{\includegraphics[width=\columnwidth]{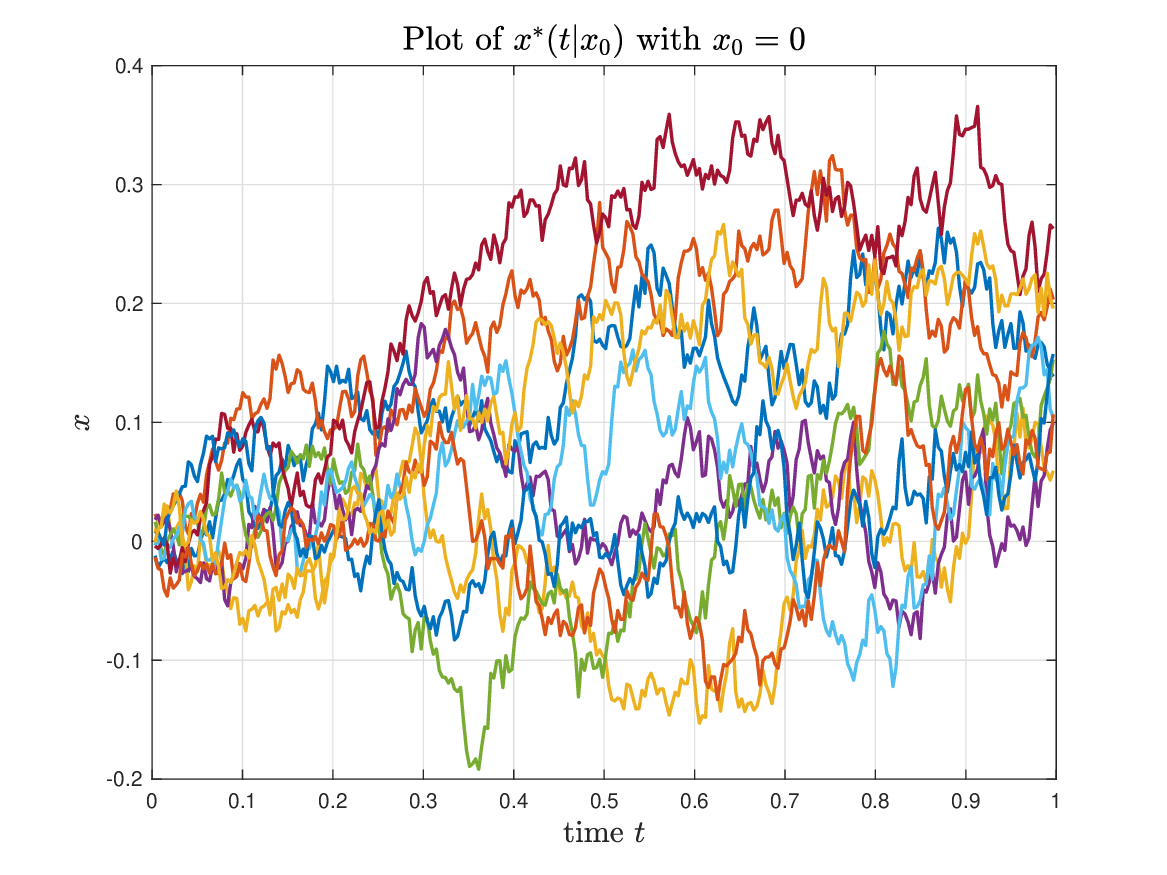}}
\caption{This is the plot of 10 sample paths characterizing the one dimensional optimal averaged state process $ x^*(t|0)$ in~\ref{eq: output_with_IC} induced by the optimal pinned control in Figure~\ref{fig0001}.} 
\label{fig: optimal averaged process}
\end{figure} 

\begin{figure}[!t]
\centerline{\includegraphics[width=\columnwidth]{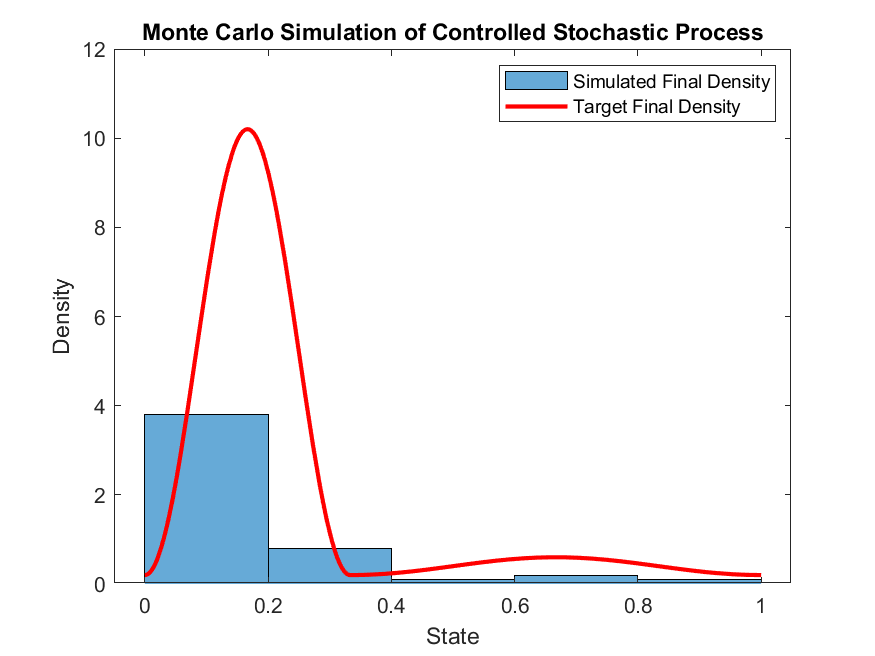}}
\caption{This is the Monte Carlo simulation that estimates  $\rho_f$ in Figure~\ref{fig_densities} using 50 states of $x_0$ randomly initialized from the support of $\rho_0$ in~\ref{fig_densities}. For such $x_0$, the control $u^*(\cdot |x_0)$ in Figure~\ref{fig0001} effectively drive the averaged process $x^*(\cdot|x_0)$ in Figure~\ref{fig: optimal averaged process} to $x^*(1|x_0)$.  Each bar's height in the histogram (in blue) corresponds to the number of samples (or frequency) that fall within the bar's range, normalized by the total number of samples and the width of the bins. Therefore, as the sample points initialized in $\rho_0$ increases, the histogram becomes a better approximation of $\rho_f$.} 
\label{fig: Monte Carlo Simulation of target density}
\end{figure} 

\begin{figure}[!t]
\centerline{\includegraphics[width=\columnwidth]{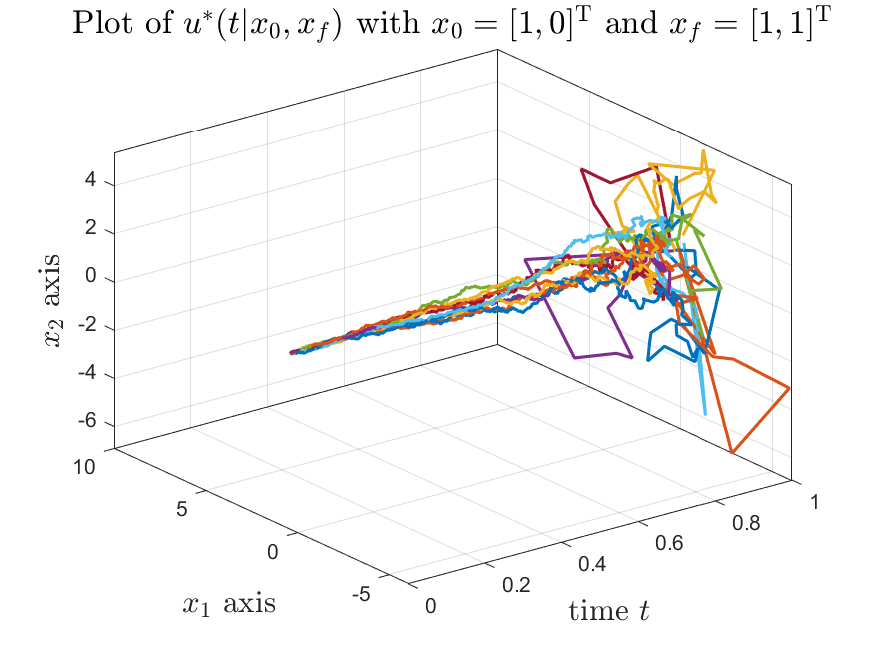}}
\caption{This is the plot 10 sample paths characterizing the two-dimensional optimal local control process $u^*(t|x_0,x_f)$ in~\eqref{eq: opt_loc_con} with parameters in~\eqref{eq: com_non_trans}-\eqref{eq: com_inv_Gram} and~\eqref{eq: int_exp_at_ft}. 
 }
\label{fig01}
\end{figure} 

\begin{figure}[!t]
\centerline{\includegraphics[width=\columnwidth]{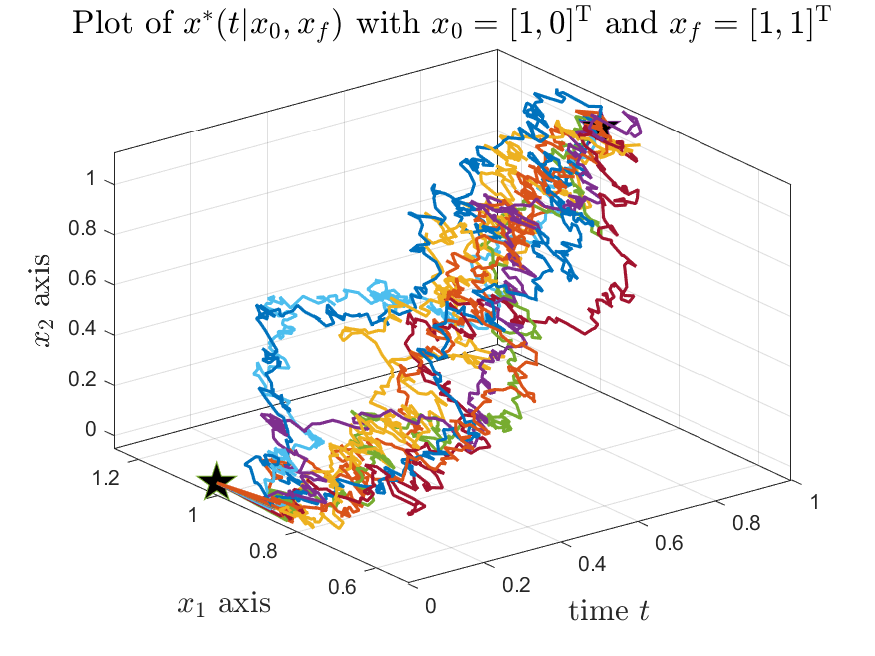}}
\caption{This is the plot of 10 sample paths characterizing the optimal averaged pinned state process $ x^*(t|x_0,x_f)$ in~\eqref{eq: pinned_output_with_IC} induced by the optimal local control in Figure~\eqref{fig01}. The lower black star indicated the initial point $[1,0]^{\mathrm{T}}$ and the upper black star indicates the point $[1,1]^{\mathrm{T}}$.  The optimal local control process in Figure~\eqref{fig01} ensured that the averaged process pinned process $ x^*(t|x_0,x_f)$ initialized at $x_0=[1,0]^{\mathrm{T}}$ (the lower black star) ends at $x_f=[1,1]^{\mathrm{T}}$ (the upper black star) almost surely at time $t=1$.} 
\label{fig02}
\end{figure}

\section{Conclusion and future work}\label{sec:Conclusion and future work}
In this paper, we have discussed the problem of conditioning a  Markov process, subjected to parameter perturbation, to initial and final given distributions. We have shown that the control that establishes the bridge is a stochastic feedforward control. This is our preliminary work and there will be more follow-up projects including discrete and computational problems. For example, although we have formally solved the optimal control problem using a path integral formulation, the task of computing the path integral remains. To address this, we can employ various standard methods, such as Monte Carlo sampling~\cite{HJK:05}. Specifically, one straightforward approach is naive forward sampling. This is our future work. Another future work is to relax the restriction that the control and the noise in our problem enter the system through the same channel. This setup deviates from being a Schrödinger bridge problem. We also plan to study the case other variations of the ensemble of stochastic process such as existence of barrier, non-linearity and extend other data-driven problems associated to the classical Schrodinger bridge problem to our case etc.  

\section*{References}

\def\refname{\vadjust{\vspace*{-2.5em}}}



\begin{thebibliography}{00}\leftskip1pc
\bibitem{SE:31}
E.~Schr{\"o}dinger, {\em {\"U}ber die Umkehrung der Naturgesetze}.
\newblock Verlag der Akademie der Wissenschaften in Kommission bei Walter De
  Gruyter u~…, 1931.

\bibitem{HF:88}
H.~F{\"o}llmer,  "Random fields and diffusion processes," {\em Lect. Notes Math}, vol.~1362,
  pp.~101--204, 1988. 

\bibitem{YC-TTG-MP:21}
Y.~Chen, T.~T.~Georgiou and M.~Pavon,  "Optimal transport in systems and control," {\em Annual Review of Control, Robotics, and Autonomous Systems}, vol.~4, no.~1,
  pp.~89--113, 2021.

\bibitem{CL:13}
C.~L{\'e}onard, "A survey of the Schrodinger problem and some of its connections with optimal transport,"{\em arXiv preprint arXiv:1308.0215}, 2013.


\bibitem{BS:32}
S.~Bernstein,  "On the connections between random quantities," {\em Verh. Boarding school. Math.-Kongr., Zurich}, vol.~1,
  pp.~288--309, 1932. 




\bibitem{DPP:91}
P.~Dai~Pra, "A stochastic control approach to reciprocal diffusion
  processes," {\em Applied Mathematics and Optimization}, vol.~23, no.~1,
  pp.~313--329, 1991.


\bibitem{PPD-PM:90}
D.~P.~Pra., M.~Pavon,'' On the Markov processes of Schr{\"o}dinger, the Feynman-Kac formula and stochastic control'', {\em Realization and Modelling in System Theory: Proceedings of the International Symposium MTNS-89, Volume I}, no.~4,
  pp.~497--504, 1990.

\bibitem{CY-GT:15}
Y.~Chen, G~Tryphon, "Stochastic bridges of linear systems,"{\em IEEE Transactions on Automatic Control}, vol.~61, no.~2,
  pp.~526--531, 2015.

\bibitem{CY-GTT-PM:15}
Y.~Chen, T.~T. Georgiou, and M.~Pavon, ``Optimal steering of a linear
  stochastic system to a final probability distribution, part i,'' {\em IEEE
  Transactions on Automatic Control}, vol.~61, no.~5, pp.~1158--1169, 2015.

\bibitem{CY-GTT-PM:16}
Y.~Chen, T.~T. Georgiou, and M.~Pavon, ``Optimal transport over a linear
  dynamical system,'' {\em IEEE Transactions on Automatic Control}, vol.~62,
  no.~5, pp.~2137--2152, 2016.  

\bibitem{HC-DL-SR-YC-SC-SI-SY-HC-IH-JK:19}
H.~Cuay{\'a}huitl, D.~Lee, S.~Ryu, Y.~Cho, S.~Choi, S.~Indurthi, and S.~Yu, H.~Choi, I.~Hwang, and J.~Kim, "Ensemble-based deep reinforcement learning for chatbots," {\em Neurocomputing},vol~366, pp.~118--130,2019.

\bibitem{JN-YK-PS:19}
J.~Nauta, Y.~Khaluf, and P.~Simoens,"Using the Ornstein-Uhlenbeck process for random exploration"{\em 4th International Conference on Complexity, Future Information Systems and Risk (COMPLEXIS2019)}, pp~59--66, 2019.

\bibitem{LJS-KN:07}
J.-S. Li and N.~Khaneja, ``Ensemble control of linear systems,'' in {\em 2007
  46th IEEE Conference on Decision and Control}, pp.~3768--3773, IEEE, 2007.

\bibitem{LJS:10}
J.-S. Li, "Ensemble control of finite-dimensional time-varying linear  systems," {\em IEEE Transactions on Automatic Control}, vol.~56, no.~2,
  pp.~345--357, 2010.

\bibitem{ADO:22}
D.~O. Adu, ``Optimal transport for averaged control,'' {\em IEEE Control
  Systems Letters}, vol.~7, pp.~727--732, 2022.


\bibitem{QJ-ZA-LJS:13}
J.~Qi, Ji and A.~Zlotnik, Anatoly and J.-S. Li,"Optimal ensemble control of stochastic time-varying linear systems."{\em Systems \& Control Letters}, vol~62, no~11, pp.~1057--1064, 2013.

\bibitem{BR-KN:00}
R.~Brockett and N.~Khaneja, "On the stochastic control of quantum ensembles,"{\em System Theory: Modeling, Analysis and Control}, pp.~75--96, 2000.

\bibitem{ADO-CY:23}
D.~O.~Adu and Y.~Chen, "Stochastic bridges over ensemble of linear systems,"{\em 2023 62nd IEEE Conference on Decision and Control (CDC)}, pp.~2803--2808, 2023


\bibitem{NE:67}
E.~Nelson, '' Dynamical theories of Brownian motion, ''  Princeton University Press, vol.~17, 1967. 

\bibitem{VHR:07}
R. V.~Handel, "Stochastic calculus, filtering, and stochastic control," {\em Course notes., URL http://www. princeton. edu/rvan/acm217/ACM217. pdf}, vol.~14, 2007.

\bibitem{KFC:12}
F.~C.~Klebaner, "Introduction to stochastic calculus with applications", World Scientific Publishing Company,2012.

\bibitem{OB:03}
B.~{\O}ksendal, "Stochastic differential equations,"  Springer, 2003.

\bibitem{RPF-FLV:00}
R.~P~Feynman and Jr.~F.~L~Vernon, "The theory of a general quantum system interacting with a linear dissipative system,"{\em Annals of physics}, vol~281, no.~1-2 pp~547--607, 2000

\bibitem{RPF-ARH-DFS:10}
R.~P~Feynman, A.~R.~Hibbs and D.~F.~Styer, "Quantum mechanics and path integrals,"  Courier Corporation, 2010

\bibitem{HJK:05}
H.~J~Kappen, "Path integrals and symmetry breaking for optimal control theory,"{\em Journal of statistical mechanics: theory and experiment}, vol~2005, no.~11 pp~11011, 2005.

\bibitem{HJK:2005}
H.~J~Kappen, "Linear theory for control of nonlinear stochastic systems,"{\em Physical review letters}, vol~95, no.~20 pp~200201, 2005

\bibitem{HJK:07}
H.~J~Kappen, "An introduction to stochastic control theory, path integrals and reinforcement learning,"{\em AIP conference proceedings}, vol~887, no.~1 pp~149--181, 2007.

\bibitem{HJK-WW-DB:07}
H.~J~Kappen, W.~Wiegerinck and B.~Van Den Broek, "A path integral approach to agent planning,"{\em Autonomous Agents and Multi-Agent Systems}, pp~41, 2007

\bibitem{DB-WW-HJK:08}
B.~Van Den Broek, W.~Wiegerinck and H.~J~Kappen, "Graphical model inference in optimal control of stochastic multi-agent systems,"{\em Journal of Artificial Intelligence Research}, vol~32, pp~95--122, 2008.

\bibitem{WW-DB-HJK:08}
 W.~Wiegerinck, B.~Van Den Broekand H.~J~Kappen, "Stochastic optimal control in continuous space-time multi-agent systems,"{\em arXiv preprint arXiv:1206.6866}, 2012.

\bibitem{AJM-HCL-AJB:90}
A.~J.~McKane, H.~C.~Luckock and A.~J.~Bray, "Path integrals and non-Markov processes. I. General formalism,"{\em Physical Review A}, vol~41, no~2, pp~644, 1990.

\bibitem{AJB-AJM-TJN:90}
A.~J.~Bray, A.~J.~McKane and T.~J.~Newman, "Path integrals and non-Markov processes. II. Escape rates and stationary distributions in the weak-noise limit,"{\em Physical Review A}, vol~41, no~2, pp~657, 1990.


\bibitem{HCL-AJM:90}
H.~C.~Luckock and A.~J.~McKane, "Path integrals and non-Markov processes. III. Calculation of the escape-rate prefactor in the weak-noise limit,"{\em Physical Review A}, vol~42, no~4, pp~1982, 1990.

\bibitem{LM-ZE:14}
M.~Lazar and E.~Zuazua, ``Averaged control and observation of
  parameter-depending wave equations,'' {\em Comptes Rendus Mathematique},
  vol.~352, no.~6, pp.~497--502, 2014.
  
\bibitem{LJ-ZE:17}
J.~Loh{\'e}ac and E.~Zuazua, ``Averaged controllability of parameter dependent
  conservative semigroups,'' {\em Journal of Differential Equations}, vol.~262,
  no.~3, pp.~1540--1574, 2017.
\bibitem{LQ-ZE:16}
Q.~L{\"u} and E.~Zuazua, ``Averaged controllability for random evolution
  partial differential equations,'' {\em Journal de Math{\'e}matiques Pures et
  Appliqu{\'e}es}, vol.~105, no.~3, pp.~367--414, 2016.
\bibitem{ZE:14}
E.~Zuazua, ``Averaged control,'' {\em Automatica}, vol.~50, no.~12,
  pp.~3077--3087, 2014.

\bibitem{NH:87}
N.~Halyo,"A combined stochastic feedforward and feedback control design methodology with application to autoland design," 1987.

\bibitem{MPS:82}
P.~S.~Maybeck,"Stochastic models, estimation, and control",1982

\bibitem{HN-DH-TDB:92}
N.~Halyo, H.~Direskeneli, and D.~B.~Taylor, "A stochastic optimal feedforward and feedback control methodology for superagility," 1992.


\bibitem{HME:89}
M.~E.~Halpen and Aeronautical Research Labs Melbourne (Australia),"Application of Optimal Tracking Methods to Aircraft Terrain Following," 1989.




\bibitem{ADO-BT-GB}
D.~O. Adu, T.~Basar and B.~Gharesifard, ``Optimal Transport for a Class of Linear Quadratic Differential Games,'' {\em IEEE Transactions on Automatic Control}, vol.~67,
  no.~11, pp.~6287-6294, 2022.

\bibitem{ADO-GB:24}
D.~O. Adu, and B.~Gharesifard, ``Robust Matching for Teams,'' {\em Journal of Optimization Theory and Applications}, vol.~200,
  no.~2, pp.~501--523, 2024.



\bibitem{CY-GTT-PM:2016}
Y.~Chen, T.~T.~Georgiou and M.~Pavan, "On the relation between optimal transport and Schr{\"o}dinger bridges: A stochastic control viewpoint,"{\em Journal of Optimization Theory and Applications}, vol~169, pp~671--691, 2016.

\bibitem{ET:90}
E.~Todorov, "Efficient computation of optimal actions,"{\em Proceedings of the national academy of sciences}, vol~106, no~28, pp~11478--11483, 2009.








\bibitem{BSP-VL:04}
S.~P.~Boyd and L.~Vandenberghe, {\em Convex Optimization}, 2004.


\bibitem{WMA:50}
M.~A.~Woodbury, {\em Inverting Modified Matrices}, 1950.

\bibitem{HHV-SSR:81}
H.~V.~Henderson and S.~R.~Searle, "On deriving the inverse of a sum of matrices,"{\em Siam Review}, vol~23, no~1, pp~53--60, 1981.
\end{thebibliography}
\end{document}